\documentclass[12pt, reqno]{amsart}
\usepackage{array}
\usepackage[utf8]{inputenc}
\usepackage{multirow}
\usepackage{tabularx}
\usepackage[utf8]{inputenc}
\usepackage{amssymb, amsmath, amsfonts, amscd, calligra, mathrsfs}
\usepackage[raiselinks=false,colorlinks=true,citecolor=blue,urlcolor=blue,
linkcolor=blue,bookmarksopen=true,pdftex]{hyperref}
\usepackage[dvipsnames]{xcolor}
\usepackage{dynkin-diagrams}
\usetikzlibrary{arrows,decorations.markings}
\usetikzlibrary{matrix}
\usetikzlibrary{graphs}
\usetikzlibrary{backgrounds}
\usepackage[title]{appendix}
\usepackage{tikz}\usepackage{tikz-cd}
\usepackage{enumerate}
\usepackage[mathscr]{eucal}
\newtheorem{theorem}{Theorem}[section]
\newtheorem{proposition}[theorem]{Proposition}
\newtheorem{lemma}[theorem]{Lemma}
\newtheorem{remark}[theorem]{Remark}

\begin{document}
\baselineskip=15.5pt
\title[Cohomology of flag bundles and Picard groups]{Cohomology of flag bunldes over compact Hermitian locally symmetric spaces}
\author[P. Biswas]{Pritthijit Biswas} 
\address{Chennai Mathematical Institute, H1 Sipcot IT Park, Siruseri, Kelambakkam 603103. Tamil Nadu, India.}
\email{pritthijit@cmi.ac.in\\
pritthibis06@gmail.com}
\author[P. Sankaran]{Parameswaran Sankaran}
\address{Chennai Mathematical Institute, H1 Sipcot IT Park, Siruseri, Kelambakkam 603103. Tamil Nadu, India.}
\email{sankaran@cmi.ac.in}
\subjclass[2010]{14M15; 17B20; 32M10; 32M15; 32Q15}
\keywords{Dolbeault cohomology, Picard group, semisimple Lie groups, Hermitian symmetric spaces, flag bundles} 
\thispagestyle{empty}
\date{}
\begin{abstract}
Let $E\to B$ be a complex analytic fiber bundle with fiber $F$, a flag variety over a compact complex manifold $B$.  We shall obtain a description 
of the cohomology of $E$ when $B=X_\Gamma:=\Gamma\backslash X, E=Y_\Gamma:=\Gamma\backslash Y$ and $F=K/H$, a flag variety, where $Y=G/H$ and $X=G/K$, a Hermitian globally symmetric space of non-compact type with $G$ being a real, connected, non-compact, semisimple linear Lie group with no compact factors and simply connected complexification,
$K\subset G$, a maximal compact subgroup, $H=Z_K(S)$, the centralizer in $K$ of a toral subgroup $S\subseteq K$ containing $Z(K)$, the center of $K$ and  
$\Gamma\subset G$, a uniform and torsionless lattice in $G$. We also obtain a description of  
the Picard group of $Y_\Gamma$ and $X_{\Gamma}$, for which the complexification of $G$ need not be simply connected.
Moreover when $G$ is simple, we obtain the values of $ q$ for which $H^{p,q}(X_\Gamma)$ 
vanishes when $p=0,1$.   
This extends the results of R. Parthasarathy from $1980$, who considered 
(partially) the case $p=0$. 
\end{abstract}
\maketitle
\section{Introduction} 
Let $G$ be a linear, connected, non-compact and real semisimple Lie group with no compact factors. Let $\mathfrak{g}_{0}$ be the Lie algebra of $G$, $\theta$ be a Cartan involution of $\mathfrak{g}_{0}$ and $\mathfrak{g}_{0}=\mathfrak{k}_{0}\oplus \mathfrak{p}_{0}$ be the associated Cartan decomposition. One has an involutive automorphism $\Theta$ of $G$ such that $\theta=d\Theta$. The fixed subgroup of $\Theta$ is a maximal compact subgroup $K$ of 
$G$ with Lie algebra $\mathfrak k_0.$  We shall denote by $\mathfrak g$, $\mathfrak k$ and $\mathfrak p$, the complexifications of $\mathfrak g_0$,$\mathfrak k_0$ and $\mathfrak p_0$ respectively and by $G_\mathbb C$ and $K_\mathbb C$, the complexifications 
of $G$ and $K$ respectively.
We assume that the symmetric space $X=G/K$
is Hermitian.  So we have a decomposition of $\mathfrak p$ into a direct sum 
$\mathfrak p=\mathfrak p^+\oplus \mathfrak p^-$ into the holomorphic and anti-holomorphic tangent spaces at the identity coset $eK$ in $G/K$.  
Note that $\mathfrak p^+$ and $\mathfrak p^-$ are $K$-representation spaces and are abelian 
subalgebras of $\mathfrak g$.  
Let $\Gamma$ be a torsion free and uniform lattice in $G$. 
We denote the double coset space $\Gamma\backslash G/K=\Gamma\backslash X$ by $X_{\Gamma}$. It is a well known result due to Kodaira \cite[Theorem (6)]{Kodaira} that $X_{\Gamma}$ is a complex projective manifold.  
Since $X$ is contractible, $X_\Gamma$ is an Eilenberg-MacLane complex 
$K(\Gamma,1)$.
Let $T\subseteq K$ be a maximal torus of $K$ and let $S\subseteq T$ be a subtorus. We will assume that $Z(K)$, the center of $K$, is contained in $S$. 
Let $H$ be the centralizer $Z_K(S)$ of $S$ in $K$.  A theorem of Borel \cite{borel-ics}
shows that $K/H$ admits a structure of homogeneous complex projective manifold.  In fact $K/H$ is a generalized complex flag manifold $K_\mathbb C/Q$ for a suitable parabolic subgroup $Q$ of $K_\mathbb{C}$, where $K_\mathbb C$ being the complexification of $K$ is complex reductive. The homogeneous space 
$G/H$ has a natural structure of a complex manifold which is $G$-invariant (for the left $G$-action) and is in fact the total space of a holomorphic $K/H$-bundle with holomorphic projection $G/H\to G/K$.  We mod out by the restricted 
$\Gamma$-action to obtain a holomorphic $K/H$-bundle $\pi: Y_\Gamma\to X_\Gamma$ 
where $Y_\Gamma=\Gamma \backslash G/H$.  Moreover in Theorem $(\ref{ygamma-is-projective})$, we conclude that $Y_\Gamma$ is a {\it complex projective manifold}. 
The purpose of this article is to study the topology and geometry of $Y_\Gamma.$
More precisely, 
we show that any complex vector bundle over $K/H\cong K_\mathbb{C}/Q$ arising from a complex representation of $H$ is in fact holomorphic and can 
be extended to a holomorphic vector bundle over $Y_\Gamma$, see Proposition $(\ref{algebraic-vector-bundles})$. 
As a consequence 
we show that when $G_{\mathbb{C}}$ is simply connected, the fiber inclusion into $Y_\Gamma$ of the bundle $\xi_{\Gamma}:=(Y_\Gamma, X_\Gamma,\pi, K/H)$, over any base point in $X_{\Gamma}$ induces a surjection in topological complex $K$-rings and rational cohomology, whereas the fiber inclusion into $Y_{\Gamma}$ over the identity double coset in $X_{\Gamma}$ induces a surjection of the Grothendieck $K$-rings. See Theorem $(\ref{fiber-restriction-is-surjectivein-ktheory})$ and Proposition $(\ref{leray-hirsch})$.  
By the Leray-Hirsch theorem, this leads to a description of the rational cohomology of $Y_\Gamma$ as 
a module over $H^*(X_\Gamma,\mathbb Q)$.   
As an application, we obtain, in Theorem $(\ref{hodge-decomposition})$, 
the following description of the Dolbeault cohomology of $Y_\Gamma$: 
\begin{center} 
$H^{p,q}(Y_\Gamma)\cong \bigoplus_{r\ge 0} H^{p-r,q-r}(X_\Gamma)\otimes H^{r,r}(K/H).$
\end{center}
Moreover we obtain a description of the Picard group of $Y_\Gamma$ in Theorem $(\ref{picard-ygamma})$, which does not require $G_{\mathbb{C}}$ to be simply connected, and states that under the assumption $H^{0,2}(X_{\Gamma})=0$ (when $\Gamma$ is irreducible, this assumption is closely related to the real rank of $G$, namely it is satisfied if real rank of $G$ is at least $3$), there is a split short exact sequence:
\begin{center}
$0\rightarrow Pic(X_{\Gamma})\overset{\pi^{*}}{\rightarrow}Pic(Y_{\Gamma})\overset{\iota^{*}}{\rightarrow}Pic(K/H)\rightarrow 0.$
\end{center}
We also provide a description of the Picard group of $X_\Gamma$, without the assumption that $G_{\mathbb{C}}$ is simply connected in Theorem $(\ref{picard-xgamma})$.  
This is perhaps well-known to the experts but we have included the details since we could not find a reference. 

Assuming that $G$ is simple, we obtain several vanishing results for $H^{p,q}(X_\Gamma)$ when $p=0,1,$ 
using the Matsushima isomorphism which expresses $H^*(X_\Gamma,\mathbb{C})$ in terms of the relative Lie 
algebra cohomology $H^*({\mathfrak g}_{0},K,A_{\mathfrak q})$ where $A_{\mathfrak q}$ denotes the 
smooth and $K$-finite vectors of the irreducible unitary representation of $G$ corresponding to a $\theta$-stable parabolic subalgebra $\mathfrak{q}$ of $\mathfrak{g}_{0}$, see the equations $(\ref{matsushima-isomorphism})$, $(\ref{hodgematsushima})$, $(\ref{gkcohomology-hodgetype})$, $(\ref{gk-cohomology-hermitiansymmetricspace})$ and the discussions in \S 3. The vanishing results for $H^{p,q}(X_{\Gamma})$ when $p=0$, which had been stated (without proof) for simple $G$ by R. Parthasarathy in \cite[\S6]{parthasarathy-1980}, have been extended in this article whereas the vanishing results for $H^{p,q}(X_{\Gamma})$ when $p=1$ seems to be new. Our vanishing results appear in \S 3, see Theorem $(\ref{vanishing})$. For simple $G$ with real rank at least $2$, Theorem $(\ref{hodge-1-1})$ in \S 3 shows that $rank(Pic(X_{\Gamma}))=1$, for all the Cartan types except possibly {\bf BDI}.
\section{Preliminaries and the main theorems}
{\bf Notations:} {\em Throughout this note, unless otherwise stated, $G$ denotes a linear, connected, non-compact and real semisimple Lie group with no compact factors, $K\subset G$ denotes a maximal compact subgroup of $G$ corresponding to the fixed point subgroup of a Cartan involution of $G$ and $G/K$ which is denoted by $X$ is assumed to be Hermitian symmetric.  We denote by 
$\Gamma\subset G$ a torsion free and uniform lattice in $G$ and by $X_\Gamma$  the locally Hermitian symmetric space 
$\Gamma\backslash X.$  If $H$ is any real Lie group, we denote by $H_\mathbb{C}$ its complexification. Moreover the real rank of $G$ is denoted by $r_{\mathbb{R}}(G)$}.  
\subsection{Cohomology of Lattices:} We shall briefly recall the relevant results on the cohomology of the compact locally Hermitian symmetric space $X_\Gamma$, where $\Gamma\subset G$ is any torsion free and uniform lattice. 
Since $X=G/K$ is contractible, $X_\Gamma$ is an Eilenberg-MacLane space $K(\Gamma,1)$ and so $H^*(X_\Gamma,\mathbb C)$ is the same as the group cohomology $H^*(\Gamma,\mathbb C).$   

Fix a measure $\mu$ on $\Gamma\backslash G$ which is invariant under the right translations by elements of $G,$ arising from a Haar measure on $G$.  One has $L^2(\Gamma\backslash G,\mu)$ as a natural unitary representation of $G$. 
  By a result of I. M. Gelfand and I. I. Piatetsky-Shapiro \cite{GGP}, $L^2(\Gamma\backslash G,\mu)$ decomposes as a discrete Hilbert space direct sum of 
irreducible unitary representations $(\pi,V_\pi)$ of $G$ each of 
which occurs with {\em finite}
multiplicity in $L^2(\Gamma\backslash G,\mu)$.  That is,
\[L^2(\Gamma\backslash G,\mu)=\widehat{\bigoplus}m(\pi,\Gamma)V_\pi.\]
where the sum is over the set $\widehat{G}$ of all (isomorphism classes) of irreducible unitary representations of $G$ and $m(\pi,\Gamma)$ denotes the (finite) multiplicity (possibly zero) with which $\pi$ occurs 
in $L^2(\Gamma\backslash G,\mu)$.  Those repesentations $\pi$ for which $m(\pi,\Gamma)>0$ are called {\em automorphic repesentations} of $\Gamma$.

Y. Matsushima \cite{matsushima} has shown that: 
\begin{equation}\label{matsushima-isomorphism}
 H^*(X_\Gamma,\mathbb C) \cong \bigoplus_{\pi\in\widehat{G}} m(\pi,\Gamma)H^*(\mathfrak g_0,K,V^\infty_{\pi,K}).
\end{equation}
Here for every $\pi\in\widehat{G}$,  $V^{\infty}_{\pi,K}$ denotes the smooth and $K$-finite vectors of $V_{\pi}$. Since $X=G/K$ is Hermitian symmetric, a classical theorem of Kodaira shows that 
$X_\Gamma$ is a complex projective manifold.  Therefore one has the Hodge decomposition $H^r(X_\Gamma;\mathbb C)\cong \bigoplus_{p+q=r}H^{p,q}(X_\Gamma)
=H^q(X_\Gamma,\Omega^p).$  It turns out, in this case, that there is a Hodge decomposition 
of the $(\mathfrak g_0,K)$-cohomology: $H^r(\mathfrak g_0,K,V^\infty_{\pi,K})
=\bigoplus_{p+q=r}H^{p,q}(\mathfrak g_0,K,V^{\infty}_{\pi,K})$.  Moreover for every $p,q\geq 0$, the Matsushima isomorphism in $(\ref{matsushima-isomorphism})$ restricts to an isomorphism:
\begin{equation}\label{hodgematsushima}
    H^{p,q}(X_\Gamma)\cong \bigoplus_{\pi\in\widehat{G}}m(\pi,\Gamma)H^{p,q}(\mathfrak g_0,K,V^\infty_{\pi,K}).
    \end{equation}
We refer the reader to \cite[Ch.I, Ch. II, \S1,\S2 and \S4]{borel-wallach}. 

Let $X_{\textrm{u}}$ be the compact dual of $X.$  Thus $X_{\textrm{u}}\cong U/K$ is a generalized complex flag manifold, where $U\subset G_\mathbb C$ is a maximal compact subgroup that contains 
$K.$  Then $H^*(X_{\textrm{u}},\mathbb C)$ is isomorphic to $H^{*}(\mathfrak g_0,K,\mathbb C)$ and moreover if $p,q\geq 0$ then $H^{p,q}(X_{\textrm{u}})=0$ unless $p=q$ and if $p\geq 0$ then $H^{p,p}(X_{\textrm{u}})\cong H^{p,p}(\mathfrak g_0,K,\mathbb C)$. For any $p\geq 0$,   one has the Matsushima homomorphism: 
\begin{equation}\label{matsushima-homomorphism}
   {j}^*:  H^{p,p}(X_{\textrm{u}})\to H^{p,p}(X_\Gamma).
\end{equation}
on identifying $H^{p,p}(X_{\textrm{u}})$ with the summand corresponding to the 
trivial (one-dimensional) representation under the Matsushima isomorphism (\ref{matsushima-isomorphism}).  
\subsection{$\theta$-stable parabolic subalgebras:}
The automorphic representations for which the modules $H^*(\mathfrak g_0,K,V_{\pi,K}^\infty)\ne 0$ 
are called the {\em cohomological representations} (for $\Gamma$).  Note that 
since $X_\Gamma$ is compact, there are only finitely many cohomological representations. Indeed, there are only finitely many irreducible unitary representations $(\pi,V_\pi)$
of $G$ with non-vanishing $H^*(\mathfrak g_0,K,V^\infty_{\pi,K}).$
It is known that they all arise from $\theta$-{\em stable parabolic subalgebras of 
$\mathfrak g_0$}.  

The $\theta$-stable parabolic subalgebras of $\mathfrak g_0$ are defined as follows. 
 Recall that $\mathfrak g_0$ is a real form of $\mathfrak g$, that is $\mathfrak g=\mathfrak g_0\otimes \mathbb C=\mathfrak g_0\oplus i\mathfrak g_0.$   This yields an involutive conjugate complex linear automorphism $\bar{}:\mathfrak g\to \mathfrak g.$  We also have 
the involution $\theta:\mathfrak g\to \mathfrak g$, which is 
the complex linear extension of the Cartan involution of $\mathfrak g_0$ that fixes $\mathfrak k_0$.  A parabolic subalgebra 
$\mathfrak q\subseteq \mathfrak g$ is called a $\theta$-{\it stable parabolic subalgebra} 
of $\mathfrak g_0$ if the following conditions are satisfied:
$(a)$  $\theta(\mathfrak q)=\mathfrak q,$  $(b)$  $\mathfrak l:=\mathfrak{q}\cap \bar{\mathfrak q}$ is a Levi subalgebra of $\mathfrak q.$ If $\mathfrak{q}$ is $\theta$-stable parabolic subalgebra of $\mathfrak{g}_{0}$, then it is well known that  $\mathfrak l=\mathfrak l_0\otimes \mathbb C$ where $\mathfrak l_0=\mathfrak l\cap \mathfrak g_0.$
The subgroup $L=\{g\in G\mid Ad(g)(\mathfrak q)=\mathfrak q\}$ is a connected closed subgroup of $G$ with $Lie(L)=\mathfrak l_0.$
We have the decomposition $\mathfrak q=\mathfrak l\oplus \mathfrak u$ where $u$ is the nil-radical of $\mathfrak q.$ It turns out that there are only finitely many $\theta$-stable parabolic subalgebras of $\mathfrak{g}_{0}$ up to $K$-conjugacy.
\subsection{Invariant complex structures on $G/H$:}\label{2.3}
Let the maximal compact $K$ of $G$ be the fixed point subgroup of a Cartan involution $\Theta:G\rightarrow G$ with $d\Theta=\theta$, a Cartan involution of $\mathfrak{g}_{0}$ giving the Cartan decomposition $\mathfrak{g}_{0}=\mathfrak{k}_{0}\oplus \mathfrak{p}_{0}$, with $\mathfrak{k}_{0}$ being the Lie algebra of $K$.
Since $X$ is Hermitian symmetric, we have a $G$-invariant complex structure 
on $X$ where the complex structure $J$ on $\mathfrak p_0$, the tangent space at the 
identity coset of $X$, equals 
$ad(X_0)$ for a suitable $X_0$ in the center $\mathfrak z(\mathfrak k_0)$ of $\mathfrak k_0$.  One has a decomposition 
$\mathfrak p=\mathfrak p^+\oplus \mathfrak p^-$ where $\mathfrak p^+,\mathfrak p^-$
are the $(+i)$- and $(-i)$-eigen spaces of $J$.  Since $X_0\in \mathfrak z(\mathfrak k_0),$ 
$\mathfrak p^+,\mathfrak p^-$ are $K$-representations and so $[\mathfrak k,\mathfrak p^\pm]\subset \mathfrak p^\pm$.  
Moreover, $[\mathfrak p^+,\mathfrak p^+]=0=[\mathfrak p^-,\mathfrak p^-]$ and $[\mathfrak p,\mathfrak p]\subseteq \mathfrak k.$   Also $\mathfrak u_0:=\mathfrak k_0\oplus i\mathfrak p_0$ is a compact Lie algebra.  One has a Lie group $U$ with Lie algebra $\mathfrak u_0$ that contains $K$.  The restriction of the involution $\theta\otimes \mathbb C$
on $\mathfrak g$ to $\mathfrak u_0$ fixes $\mathfrak k_0$.  The homogeneous space 
$X_\mathrm{u}:=U/K$ is a Hermitian symmetric space of compact type, which is  
the {\em compact dual} of $X$ and is a generalized complex flag manifold.

Let $S\subseteq T$ be a subtorus of $T$ containing $Z(K)$, the center of $K$, where $T\subseteq K$ is a maximal torus of $K$. Let $\mathfrak{s}_{0}$ and $\mathfrak{t}_{0}$ denote the Lie algebras of $S$ and $T$ and their complexifications are denoted by $\mathfrak{s}$ and $\mathfrak{t}$ respectively. Note that $\mathfrak{t}$ is a Cartan subalgebra of $\mathfrak{g}$.  
Let $H=Z_K(S),$ the centralizer of $S$ in $K$ and we will denote by $\mathfrak{h}_{0}$ the Lie algebra of $H$ and its complexification by $\mathfrak{h}$.  
It is known that $K/H$ admits a $K$-invariant complex structure. (See \cite{borel-ics}).
In fact, 
$K/H$ is biholomorphic to a generalized complex flag manifold $K_\mathbb C/Q$ for a suitable parabolic subgroup $Q$ of $K_{\mathbb{C}}$ whose Levi part is $H_\mathbb C$.  
Let $Y=G/H$ and if $\Gamma$ is a torsion free and uniform lattice in $G$, then we will denote by $Y_{\Gamma}$ the space $\Gamma\backslash Y$.   

The following theorem is well-known from the work of Borel \cite{borel-ics} who classified all homogeneous spaces which admit invariant complex structures. See 
also \cite{wang} and \cite{nishiyama}.
\begin{theorem} \label{borel-noncompactbundle}
Let $H\subseteq K\subseteq G$ be as above.\\ 
$(i)$  The homogeneous space $G/H$ admits a $G$-invariant Kähler structure such that 
    the natural quotient map $G/H\to G/K$ is the projection of a holomorphic bundle with fiber being biholomorphic to the generalized complex flag manifold $K/H$.   \\
    $(ii)$  If $\Gamma\subset G$ is a torsion free and uniform 
    lattice in $G$, then $Y_\Gamma$ admits the structure of a  Kähler manifold 
    such that the natural quotient 
    $Y_\Gamma\to X_\Gamma$ is a holomorphic $K/H$-bundle. 
 \end{theorem}    
As we would need in the sequel, we quote the following useful criterion for a homogeneous 
space $G/H$ to admit a $G$-invariant complex structure. See \cite[Proposition (6.1)]{vogan2}.

\begin{theorem} \label{criterion-ics}
 Let $G$ be a linear connected Lie group and let $H\subseteq G$ 
be a closed Lie subgroup.
    The coset space $G/H$ admits a $G$-invariant complex structure if there is a 
    subalgebra $\mathfrak r\subset \mathfrak g$ such that 
    $(i)$ $\mathfrak r\cap \bar{\mathfrak r}=\mathfrak h,$ $\mathfrak r
    +\bar{\mathfrak r}=\mathfrak g$ and
    $(ii)$ the adjoint action of $H$ on $\mathfrak g$ preserves $\mathfrak r.$
\end{theorem}
The second condition in the above theorem holds trivially if $H$ is connected. For the convenience of the reader, in what follows, we would show that $G/H$ admits a $G$-invariant complex structure with respect to which $Y_{\Gamma}\rightarrow X_{\Gamma}$ turns out to be a projection of a holomorphic fiber bundle with fiber $K/H$ and structure group a suitable parabolic subgroup of $G_{\mathbb{C}}$. We would further show that $Y_{\Gamma}$ is a complex projective manifold.
 
Let $\Delta$ denote the roots of $\mathfrak g$ with respect to $\mathfrak t.$
Denote by $\Delta_{K}$ the set of compact roots and by $\Delta_{\textrm{n}}$ the set of non-compact roots.  Thus $\mathfrak k=\mathfrak t\oplus (\bigoplus_{\alpha\in \Delta_K}\mathfrak g_\alpha)$ and $\mathfrak p=\bigoplus_{\alpha\in \Delta_{\textrm{n}}}\mathfrak g_\alpha,$ where 
$\mathfrak g_\alpha$ is the $\alpha$-root space of $\mathfrak g$. Let $\Delta_{H}=\{\alpha\in\Delta_{K}\mid \mathfrak{g}_{\alpha}\subseteq \mathfrak{h}\}$. Note that as $\mathfrak{h}$ is complex and $\overline{\mathfrak{g}_{\alpha}}=\mathfrak{g}_{-\alpha}$ for any $\alpha\in\Delta$, $\Delta_{H}=-\Delta_{H}$. Moreover as $\mathfrak{h}$ is the centralizer of $\mathfrak{s}$ in $\mathfrak{k}$, we have the decomposition: 
\[\mathfrak{h}=\mathfrak{t}\oplus(\bigoplus_{\alpha\in \Delta_{H}}\mathfrak{g}_{\alpha}).\]
Denote by $C$ the set of complementary roots 
$\Delta_K\setminus \Delta_H.$ 
Thus: 
\[\mathfrak k=\mathfrak h\oplus(\bigoplus_{\alpha\in C} \mathfrak g_\alpha).\] 
Let $\mathfrak{t}_{\mathbb{R}}:=\{X\in\mathfrak{t}\mid \alpha(X)\in\mathbb{R}~\forall~\alpha\in\Delta\}$, then $\mathfrak{t}_{\mathbb{R}}$ is a real form of $\mathfrak{t}$ and as every root is purely immaginary on $\mathfrak{t}_{0}$, we have $\mathfrak{t}_{\mathbb{R}}=i\mathfrak{t}_{0}$.

We claim that there is a positive system $\Delta^{+}$ for $\Delta$ which satisfies: $\alpha\in C^+(:=\Delta^{+}\cap C)\implies \alpha>\beta ~ \forall ~ \beta \in \Delta_{H}^{+}(:=\Delta_{H}\cap \Delta^{+})$, where for any $\alpha,\beta\in\Delta$, by $\alpha>\beta$, we mean that if $\alpha-\beta\in\Delta$ then $\alpha-\beta\in\Delta^{+}$. Once this is done, then $\mathfrak{r}:=\mathfrak{h}\oplus \mathfrak{n}^{+}\oplus \mathfrak{p}^{+}$ satisfies $\mathfrak{r}\cap \bar{\mathfrak{r}}=\mathfrak{h}$ and $\mathfrak{r}+\bar{\mathfrak{r}}=\mathfrak{g}$ where  $\mathfrak{n}^{\pm}:=\bigoplus_{\alpha\in\pm C^{+}}\mathfrak{g}_{\alpha}$ which satisfy $\overline{\mathfrak{n}^{+}}=\mathfrak{n}^{-}$. Consequently by Theorem $(\ref{criterion-ics})$, $G/H$ admits a $G$-invariant complex structure.

{\it Proof of claim:} It is known that there exists a non-zero element $X_{0}\in \mathfrak{z}(\mathfrak{k}_{0})$, the center of $\mathfrak{k}_{0}$ such that the $G$-invariant complex structure on $G/K$ is given by $ad_{X_{0}}|\mathfrak{p}_{0}$. Fix a real basis $\{H_{1},...,H_{s}\}$ of $\mathfrak{t}_{\mathbb{R}}\cap \mathfrak{s}$ with $H_{1}=-iX_{0}$, and extend it to a real basis $\{H_{1},...,H_{s},H_{s+1},...,H_{n}\}$ of $\mathfrak{t}_{\mathbb{R}}$. Now define $\Delta^{+}:=\{\alpha\in\Delta \mid \exists$ $k$ with $1\leq k\leq n$ which satisfies $\alpha(H_{i})=0$ for all $1\leq i\leq k-1$ and $\alpha(H_{k})> 0\}$. Then one checks that $\Delta^{+}$ is in fact a positive system for $\Delta$. The fact $\mathfrak{s}=(\mathfrak{s}\cap\mathfrak{t}_{\mathbb{R}})\oplus(\mathfrak{s}\cap i\mathfrak{t}_{\mathbb{R}})$ and our choice of positivity forces that $\alpha\in C^{+}\implies \alpha>\beta$ for all $\beta \in \Delta_{H}^{+}$. Indeed if $\alpha\in C^{+}$ and $\beta \in \Delta_{H}^{+}$, then $\exists$ $j\in \{1,....,s\}$ with $\alpha(H_{j})>0$ and $\alpha(H_{i})=0$ for all $1\leq i<j$. If this is not the case then $\alpha(\mathfrak{s}\cap \mathfrak{t}_{\mathbb{R}})=\{0\}$ which implies $\alpha(\mathfrak{s})=\{0\}$ ($\alpha$ is $\mathbb{C}$ linear and $\mathfrak{s}\cap\mathfrak{t}_{\mathbb{R}}$ is a real form of $\mathfrak{s}$). This means that $\mathfrak{g}_{\alpha}\subseteq \mathfrak{h}$, which is a contradiction because $\alpha\in C^{+}$. Now $\beta$ being an element of $\Delta_{H}^{+}$, is identically zero on $\mathfrak{s}\cap\mathfrak{t}_{\mathbb{R}}$. Thus if $\alpha-\beta$ is a root, it is a positive root in $\Delta_{K}$. 
Note that as $H_{1}=-iX_{0}$, we have $\mathfrak{p}^{\pm}=\bigoplus_{\alpha\in\Delta_{\textrm {n}}^{\pm}}\mathfrak{g}_{\alpha}$, where $\Delta_{\textrm{n}}^{\pm}:=\Delta_{\textrm{n}}\cap\pm\Delta^{+}$.  $\hfill$ $\Box$

Note that the subalgebra $\mathfrak{q}:=\mathfrak{h}\oplus\mathfrak{p}^{+}$ of $\mathfrak{k}$ satisfies $\mathfrak{q}\cap\bar{\mathfrak{q}}=\mathfrak{h}$ and $\mathfrak{q}+\bar{\mathfrak{q}}=\mathfrak{k}$. Hence by Theorem $(\ref{criterion-ics})$, $K/H$ has a $K$-invariant complex structure.

We now show that $K/H\rightarrow G/H\rightarrow G/K$ is a holomorphic fiber bundle. It is already clear that the inclusion of the fiber over the identity coset in $G/K$ and the projection $G/H\rightarrow G/K$ are holomorphic. We will stick to the above positive system for $\Delta$.  
It is clear that 
$\mathfrak n^\pm=\bigoplus_{\alpha\in \pm C^{+}} \mathfrak g_\alpha$ are nilpotent subalgebras of $\mathfrak k$ and $[\mathfrak h,\mathfrak n^\pm]\subseteq \mathfrak n^\pm$.  Thus $\bar{\mathfrak q}=\mathfrak h\oplus \mathfrak n^-$ is a parabolic subalgebra of $\mathfrak k$ and so $K_\mathbb C/Q$ is a generalized complex flag manifold, where $Lie(Q)=\mathfrak q$. It is known that $K$ acts transitively on $K_\mathbb C/Q$ so that $K/H\cong K_\mathbb C/Q$. 
By \cite[Proposition (7.14), \S7, Chapter VIII]{helgason}, we have an embedding  $G/K\subset G_\mathbb C/P$ where $P$ is the parabolic subgroup of $G_{\mathbb{C}}$ with Lie algebra $\mathfrak k\oplus \mathfrak p^-.$   Let $R\subset P$ be the parabolic subgroup 
of $G_\mathbb C$ with Lie algebra $\bar{\mathfrak r}=\mathfrak{h}\oplus\mathfrak{n}^-\oplus\mathfrak{p}^-.$  Then $R\cap K_\mathbb C=Q$ and 
we have 
$R\cap G=R\cap P\cap G=R\cap K=R\cap K_\mathbb C\cap K=Q\cap K=H$.
It follows 
that the quotient map $G/H\to G/K$ is the projection of a 
holomorphic bundle with fiber $K/H\cong K_\mathbb C/Q\cong P/R$ and holomorphic structure group $P$, as this is just 
the restriction to $G/K$ of the holomorphic $P/R$-bundle with projection $G_\mathbb C/R\to G_\mathbb C/P$ and holomorphic structure group $P$. 
\begin{remark}
There are only finitely many invariant complex structures on $K/H.$ 
See \cite{wang}.
Moreover, there is a natural bijective correspondence between the collection of 
$K$-invariant complex structures on $K/H$ and $G-$invariant complex structures on $G/H$ as $X=G/K$ is Hermitian symmetric. 
\end{remark}
\begin{proposition}\label{Ygamma projective}
Keeping the above notations, the space $Y_\Gamma$ admits a structure of a smooth complex projective variety. 
\end{proposition}
\begin{proof}
Let $p:\widetilde{G_{\mathbb{C}}}\rightarrow G_{\mathbb{C}}$ be a holomorphic covering of $G_{\mathbb{C}}$ with $\widetilde{G_{\mathbb{C}}}$ being simply connected. Then $\widetilde{R}:=p^{-1}(R)$ is a parabolic subgroup of $\widetilde{G_{\mathbb{C}}}$ as $R$ is a parabolic subgroup of $G_{\mathbb{C}}$. It is well known that there exists an irreducible finite dimensional complex linear representation $(V,\rho)$ of $\widetilde{G_{\mathbb{C}}}$ and a non-zero vector $v$ of $V$ such that the stabilizer of $[v]\in \mathbb{P}(V)$ under the action of $\widetilde{G_{\mathbb{C}}}$ on $\mathbb{P}(V)$ induced by $\tilde{\pi}\circ\rho$ is same as $\widetilde{R}$, where $\tilde{\pi}:GL(V)\rightarrow PGL(V)$ is the natural projection. As $(V,\rho)$ is irreducible $\widetilde{G_{\mathbb{C}}}$-representation and $ker(p)\subseteq Z(\widetilde{G_{\mathbb{C}}})$, by Schur's lemma, the morphism of complex Lie groups $\tilde{\pi}\circ \rho:\widetilde{G_{\mathbb{C}}}\rightarrow PGL(V)$ descends to a morphism $\sigma:G_{\mathbb{C}}\rightarrow PGL(V)$, inducing a $G_{\mathbb{C}}$-action on $\mathbb{P}(V)$ which obviously restricts to a $P$-action on $\mathbb{P}(V)$. This shows that via the biholomorphism $\widetilde{G_{\mathbb{C}}}/\widetilde{R}\cong G_{\mathbb{C}}/R$, we get a $G_{\mathbb{C}}$-equivariant holomorphic embedding $G_{\mathbb{C}}/R\hookrightarrow \mathbb{P}(V)$ which restricts to a $P$-equivariant holomorphic embedding $P/R\hookrightarrow \mathbb{P}(V)$.

Let $E\rightarrow X_{\Gamma}$ denote the holomorphic principle $P$-bundle corresponding to the holomorphic fiber bundle $K/H\cong P/R\rightarrow Y_{\Gamma}\rightarrow X_{\Gamma}$ with holomorphic  structure group $P$. Then using the homomorphism $\sigma|_{P}:P\rightarrow PGL(V)$, we get that $E\times_{P}PGL(V)$ is a holomorphic principle $PGL(V)$-bundle over $X_{\Gamma}$. Note that there is an isomorphism $(E\times_{P}PGL(V))\times_{PGL(V)}\mathbb{P}(V)\cong E\times_{P}\mathbb{P}(V)$ of holomorphic fiber bundles with fiber $\mathbb{P}(V)$ and holomorphic structure group $PGL(V)$ over $X_{\Gamma}$. Thus as $X_{\Gamma}$ is a smooth complex projective variety \cite[Theorem (6)]{Kodaira}, by \cite[Theorem (8)]{Kodaira},  $E\times_{P}\mathbb{P}(V)$ is also a smooth complex projective variety. As there is an isomorphism $E\times_{P}P/R\cong Y_{\Gamma}$ of holomorphic fiber bundles over $X_{\Gamma}$ with fiber $P/R$ and holomorphic structure group $P$, the $P$-equivariant holomorphic embedding $P/R\hookrightarrow \mathbb{P}(V)$, obtained at the end of the earlier paragraph, shows that there is a natural holomorphic embedding $Y_{\Gamma}\hookrightarrow E\times_{P}\mathbb{P}(V)$ forcing $Y_{\Gamma}$ to be a smooth complex projective variety.
\end{proof}
We summarize the discussions so far in \S $(\ref{2.3})$ as the following theorem:
\begin{theorem} \label{ygamma-is-projective}
Let $X=G/K$ be a Hermitian symmetric domain of non-compact type where 
$G$ is a linear, connected, non-compact and real semisimple Lie group with no compact factors, and $K$ is a maximal compact subgroup of $G$. Let $H$ be the $K$-centralizer of a torus in $K$ containing the center of $K$.
Let $\Gamma$ be a torsion less and uniform lattice in $G$, $Y_{\Gamma}:=\Gamma\backslash G/H$ and $X_{\Gamma}:=\Gamma\backslash G/K$.
    Then $Y_\Gamma$ is 
    a smooth complex projective variety, which is the total space of holomorphic fiber bundle $\xi_\Gamma$ with holomorphic projection $\pi: Y_\Gamma\to X_\Gamma$, fiber space a generalized complex flag manifold biholomorphic to $K/H$ and  holomorphic structure group a parabolic subgroup of $G_{\mathbb{C}}$. $\hfill$ $\Box$
\end{theorem}
\begin{remark}
    At this point we would like to mention that the holomorphic fiber bundle $\xi_{\Gamma}$ with fiber $K/H\cong P/R$ and holomorphic structure group $P$ may not necessarily admit a holomorphic reduction of structure group to $K_{\mathbb{C}}$ or a smooth reduction of structure group to $K$, although $K/H\rightarrow Y_{\Gamma}\rightarrow X_{\Gamma}$ is in its own right a smooth fiber bundle with fiber $K/H$ and smooth structure group $K$.    
\end{remark}
\subsection{Cohomology and Picard group of $Y_\Gamma$:} Since $Y_\Gamma$ is Kähler, 
we have the Hodge decomposition $H^r(Y_\Gamma,\mathbb C)=\bigoplus_{p+q=r}H^{p,q}(Y_\Gamma).$  The (singular) cohomology of $Y_\Gamma$ may be computed using the Leray-Serre spectral sequence of the bundle $\xi_\Gamma$, whereas the Dolbeault cohomology of $Y_{\Gamma}$ may be computed using the Borel spectral sequence of the bundle $\xi_{\Gamma}$. In the sequel we prove that when $G_{\mathbb{C}}$ is simply connected, the Leray-Hirsch theorem applies for the bundle $\xi_{\Gamma}$ with rational coefficients, from which we deduce a decomposition of the Dolbeault cohomology of $Y_{\Gamma}$ in terms of the same for its base and fiber spaces. We begin with the following remark, which points out a special case when the Leray-Hirsch theorem applies to $\xi_{\Gamma}$ with integer coefficients.
\begin{remark}
Suppose that $K/H$ is a product of flag manifolds of type $A$.  Thus, $K/Z(K)$ is locally isomorphic to $\prod_{1\le j\le k}
SU(n_j).$ Then a more refined argument yields an isomorphism with integer coefficients: 
\[H^{r}(Y_\Gamma,\mathbb Z)\cong \bigoplus_{p+q=r}H^p(X_\Gamma,\mathbb Z)\otimes H^q(K/H,\mathbb Z).\]
Note that our hypothesis on $K$ is equivalent to the requirement that each simple factor of 
$G$ be locally isomorphic to one of the groups 
$SU(p,q),Sp(n,\mathbb R), SO^*(2n)$ or $ SO_0(2, 3).$
\end{remark} 
We now turn to the rational cohomology of $Y_\Gamma$ in the general case.
We start with the following proposition:  
\begin{proposition} \label{algebraic-vector-bundles}
    Any complex vector bundle $\mathcal V$ over $K/H$, arising from 
    a representation $\rho: H\to GL(V),$
    can be extended to a holomorphic vector bundle $\mathcal V_\Gamma$ over $Y_\Gamma.$  
    \end{proposition}
\begin{proof} Let $\tilde{\rho}:H_{\mathbb{C}}\rightarrow GL(V)$ denote the unique holomorphic representation of $H_{\mathbb{C}}$ extending $\rho$. Since $H_\mathbb C$ is a Levi subgroup of $R$, $\tilde{\rho}$ further extends to a holomorphic representation $\rho':R\to GL(V)$ where $\rho'$ is trivial on the 
unipotent radical of $R$.
Therefore we obtain a holomorphic vector bundle $\mathcal V':=G_\mathbb C\times_{\rho'} V$ over $G_{\mathbb{C}}/R$.
This bundle restricts to a holomorphic vector bundle $\widetilde{ \mathcal V}\cong G\times_{\rho}V$
over $G/H$ via the holomorphic inclusion $G/H\hookrightarrow G_\mathbb C/R.$  We mod out by the left action 
of $\Gamma$ on the total space of $\widetilde{\mathcal V}$ to obtain a holomorphic 
vector bundle $\widetilde{\mathcal V}_\Gamma$ over $Y_{\Gamma}$. It is clear that $\widetilde{\mathcal V}_\Gamma$ restricts 
to $\mathcal V=K\times_{\rho}V$ over $K/H$ via the holomorphic inclusion of fiber over the identity coset in $X_{\Gamma}$. This also shows that $\mathcal{V}$ has a structure of a holomorphic vector bundle over $K/H$.
\end{proof}
Suppose that $p:\widetilde {G_\mathbb C}\to G_\mathbb C$ is a covering projection with $\widetilde {G_\mathbb C}$ connected. Let $\widetilde G\subset \widetilde {G_\mathbb C}$ be the Lie subgroup of $\widetilde {G_\mathbb C}$ corresponding to the real form $\mathfrak g_0\subset\mathfrak g$.  Then $p_0:\widetilde G\to G$ is a covering projection.  The subgroup $\widetilde K:=p_0^{-1}(K)$ is connected and is a maximal compact 
subgroup of $\widetilde G.$  Thus $X=G/K\cong \widetilde G/\widetilde K.$

\begin{proposition} \label{simplyconnectedK} We keep the above notations. Suppose that $\widetilde{ G_{\mathbb{C}}}$ is simply connected. Then\\ 
    $(i)$ The subgroup $\widetilde {K_\mathbb C}\subset 
    \widetilde {G_\mathbb C}$ corresponding to the Lie subalgebra $\mathfrak k\subset \mathfrak g$ 
    is the complexification of $\widetilde K$ and $\pi_1(\widetilde K)$ is a free abelian group.\\  
   $(ii)$ Let $S\subseteq T$ be a subtorus that contains $Z(K), H=Z_K(S),$ and let 
$\widetilde{S}$
     be the identity component of $p_0^{-1}(S)\subset \widetilde T$ where  
 $\widetilde {T}:=p_0^{-1}(T).$   Then $\widetilde{T}$ is a maximal torus of $\widetilde K$,  $\widetilde{S}$ contains $Z_{0}(\widetilde{K})$, 
    the identity component of $Z(\widetilde{K})$.
     The group $\widetilde H:=p_0^{-1}(H)$ is connected and equals $Z_{\widetilde K}(\widetilde S)$ and 
     $\pi_1(\widetilde H)$ is a free abelian group.
    \end{proposition}
\begin{proof}
   $(i)$ Note that $\widetilde {K_\mathbb C}$ is the complexification of $K$  
    by \cite[Theorem $(27.1)$]{bump}.  It remains to show that $\pi_1(\widetilde K)$ has no torsion.

    Let $U\subset \widetilde {G_\mathbb C}$ be the maximal compact subgroup of $\widetilde {G_\mathbb C}$ 
    that contains $\widetilde K.$  Then $X_{\textrm{u}}:=U/\widetilde K$ is the compact dual of $X=G/K\cong\widetilde G/\widetilde K.$  Since $\widetilde {G_\mathbb C}$
    is simply-connected, so is $U$. Since $H_2(X_{\textrm{u}},\mathbb Z)\cong \pi_2(X_{\textrm{u}})$ 
    has no torsion, using the homotopy exact sequence associated to the principal 
    $\widetilde K$-bundle with projection $U\to X_{\textrm{u}}$ we see that $\pi_1(\widetilde K) $
    has no torsion.

    $(ii)$  It is easily seen that $\widetilde{T}$ is a maximal torus of $\widetilde{K}$, $\widetilde H$ is connected, $Z_{0}(\widetilde{K})\subseteq \widetilde{S}$, 
    and that     
    $\widetilde H=Z_{\widetilde K}(\widetilde S)$.  Now $F:=K/H\cong \widetilde K/\widetilde H$ is simply connected and is also a generalized 
    complex flag manifold. So $\widetilde{H}$ is connected. Using homotopy exact sequence of the principal $\widetilde H$-bundle with projection 
    $\widetilde K\to F,$ the fact that $\pi_2(F)$ is free abelian, and that $\pi_1(\widetilde{K})$ is free abelian by $(i)$,  it follows that $\pi_1(\widetilde{H})$ is also free abelian.  
    \end{proof}
\begin{remark}\label{gc-simplyconnected}
If $\Gamma\subset G$ is a torsion free lattice, then $p_0^{-1}(\Gamma)=\widetilde \Gamma$ is a lattice 
in $\widetilde G$ but contains the kernel of $p_{0}$, which is a finite group, and so is not torsion free.  Although $\Gamma\backslash G/H\cong \widetilde \Gamma \backslash\widetilde G/\widetilde H,$ in order to prove surjectivity of the homomorphism $j^!:\mathcal K(Y_\Gamma)\to \mathcal K(F)$ in Theorem $(\ref{fiber-restriction-is-surjectivein-ktheory})$, we need to apply Proposition $(\ref{algebraic-vector-bundles})$ for the lattice $\widetilde \Gamma.$ 
For this reason, we shall assume, up to the end of Theorem $(\ref{hodge-decomposition})$, that the complexification $G_\mathbb C$ of $G$ is simply connected.  Then the proof of Proposition $(\ref{simplyconnectedK})$ shows that $\pi_1(K), \pi_1(H)$ are free abelian.  
\end{remark}
  We recall the description of the topological complex $K$-ring $K(F)$ where $F=K/H$ is a complex 
  flag manifold.  Let $RH$ denote the complex representation ring of $H$. (See \cite{husemoller}). One has the restriction homomorphism $\rho :RK\to RH$ which makes $RH$ into a $RK$-module.  We have the augmentation $RK\to \mathbb Z$ which maps $[V]$ to $\dim V$ for any continuous complex representation $V$ of $K.$
  If $V$ is a continuous complex representation of $H$, we have a vector bundle $K\times_HV\to F$.
  This defines a ring homomorphism $\rho:RH\to K(F).$  
  Since $\pi_1(H)$ has no torsion, the work of Atiyah and Hirzebruch \cite{atiyah-hirzebruch} 
  and Pittie \cite{pittie} show that $\rho$ defines an isomorphism of rings 
  $RH\otimes_{RK}\mathbb  Z\to K(F)$.  
  The ring $RH\otimes_{RK} \mathbb Z$ is the quotient $RH/J$ of $RH$ modulo the ideal generated 
  by elements $[V]-\dim V$ where $V$ is the restriction to $H$ of a continuous complex representation 
  of $K$.

Let $Z$ be a complex projective manifold. Denote by $\mathcal K(Z)$ the Grothendieck $K$-ring of {\it holomorphic} vector bundles over $Z$.  
When $Z=F$, a generalized complex flag manifold, the natural `forgetful' homomorphism $\mathcal K(F)\to K(F)$ is an 
isomorphism of rings. This a consequence of the {\it cellular filtration lemma}---see \cite[Proposition (5.5.6)]{chriss-ginzburg}, applied to the situation where, in their notation, $G$
is the trivial group and $X$ is a point.
Therefore we obtain an isomorphism $ \overline{\rho}: RH/J\to \mathcal K(F).$

Explicitly, if 
 $\sigma_0:H\to GL(V)$ is homomorphism, we obtain a unique complex analytic extension $\sigma:H_\mathbb C\to GL(V)$ of $\sigma_0$,  which then 
extends to a 
complex analytic representation $\tilde\sigma:Q\to GL(V)$, as $H_{\mathbb{C}}$ is a Levi subgroup of $Q$.  We therefore obtain a complex analytic 
vector bundle $\mathcal V=K_\mathbb C\times_{\tilde \sigma} V$ over $ K_\mathbb C/Q$. 
Under the natural identification $F=K/H\cong K_\mathbb C/Q,$ the isomorphism $\bar\rho$
sends $[V]+J\in RH/J$  to $[\mathcal V]$ in $\mathcal K(F).$

For any $b\in X_\Gamma$, the inclusion of a fiber $j_b:F_b\hookrightarrow Y_\Gamma$ yields the homomorphism $j_b^!:\mathcal K(Y_\Gamma)\to \mathcal K(F_b)\cong \mathcal K(F)$, via restriction.  We shall take 
$b\in X_\Gamma$ to be $\Gamma K$, the trivial double coset, so that $F_b=F=K/H\cong K_\mathbb C/Q.$
 The following theorem follows from 
 Proposition $(\ref{algebraic-vector-bundles})$, the fact that every holomorphic line bundle on $K_{\mathbb{C}}/Q$ arises from a holomorphic character of $Q$ which is trivial on $Q_{\textrm{u}}$, the unipotent part of $Q$, and the above discussion.  
 \begin{theorem} \label{fiber-restriction-is-surjectivein-ktheory} Let $F=K/H$. 
The homomorphism
$j^!:\mathcal K(Y_\Gamma)\to \mathcal K(F)$ 
is a surjection. Moreover 
$j^!:Pic(Y_\Gamma)\to Pic(F)$ and $j^*:H^2(Y_\Gamma,\mathbb Z)\to H^2(F,\mathbb Z)$
are surjections.
\hfill $\Box$
\end{theorem}
\begin{remark}\label{picard-surj}
The surjectivity of $j^{!}:Pic(Y_\Gamma)\rightarrow  Pic(F)$ and therefore that of $j^{*}:H^2(Y_\Gamma,\mathbb Z)\rightarrow H^2(F,\mathbb Z)$ in the above theorem hold without the assumption that 
$G_\mathbb C$ is simply connected. 
\end{remark}
Next we consider the behaviour of topological complex $K$-ring and the rational cohomology under the restriction to the fiber over any base point in $X_{\Gamma}$.
\begin{proposition}\label{leray-hirsch}
    For any $x\in X_\Gamma$, the fiber inclusion $j_x:F_x\hookrightarrow Y_\Gamma$ induces surjections:
    $(i)$ $j_x^!:K(Y_\Gamma)\to K(F_x)$ and  
     $(ii)$ $j_x^*:H^*(Y_\Gamma,\mathbb Q)\to H^*(F_x,\mathbb Q)$.
\end{proposition}
  \begin{proof} When $x=b:=\Gamma K=[e]\in X_\Gamma$ the trivial coset, where $e$ is the identity element of $G$, assertion $(i)$ is valid by Theorem $(\ref{fiber-restriction-is-surjectivein-ktheory})$. 
         For an aribtrary $x$,  we join $b$ to $x$ by a path $\sigma$. Then we claim that there exists a homotopy $f: F\times I\to Y_\Gamma$ such that, for each $t\in I, f$ defines a biholomorphism $f_t:F\times\{t\}\to F_{\sigma(t)}$ onto the fiber over $\sigma(t).$
         To see this, note that the image of $\sigma$ can be covered by {\em finitely many}  charts $(U_j,\psi_j)$ that trivialize the $F$-bundle $\xi_\Gamma$. 
         So, it suffices to establish the claim for any path joining two points $x_0,x_1$ both of which belong to a trivializing open set $U\subset X_\Gamma$. Then the above $f$ can be constructed inductively by suitably exploiting the transition functions of the bundle $\xi_{\Gamma}$.
                   Recall that $\pi:Y_\Gamma\to X_\Gamma$ denotes the bundle projection. 
          Let $\lambda:I\to U$ be a path from $x_0$ to $x_1$.   We then have a homotopy $f: F\times I\to \pi^{-1}(U)$ such that 
         $f_t, t\in I,$ is a biholomorphism. Explicitly, if $\psi: F\times U\to \pi^{-1}(U)$ is a biholomorphism, then we define
         $f(x,t)=\psi(x,\lambda(t))$.  This proves our claim. 
         It follows that $j_x^!: K(Y_\Gamma)\to K(F_x)$ is a surjection since $j_b^!$ is. Thus $(i)$ follows.   
         
         Assertion $(ii)$ follows from $(i)$ and the naturality of the Chern character.  Indeed for any $x\in X_{\Gamma}$, the Chern character $ch_{F_{x}}: K(F_{x})\otimes \mathbb Q\to H^*(F_{x},\mathbb Q)$ is an isomorphism, by what has been proved in $(i)$, 
         $j_{x}^!:K(Y_\Gamma)\otimes \mathbb Q\to K(F_{x})\otimes \mathbb Q$ is surjective. Since $H^*(j_{x},\mathbb Q)\circ ch_{Y_\Gamma}=ch_{F_{x}}\circ j_{x}^!$, it follows that 
        $j_{x}^*:H^*(Y_\Gamma,\mathbb Q)\to H^*(F_{x},\mathbb Q)$ is surjective.    \end{proof}

It is evident from the proof of $(i)$ in Proposition $(\ref{leray-hirsch})$, that for any $p\geq 0$, if a finite subset $S$ of $H^p(Y_{\Gamma},\mathbb{Q})$ restricts to a $\mathbb{Q}$-basis of $H^{p}(F_{b},\mathbb{Q})$, then $S$ restricts to a $\mathbb{Q}$-basis of $H^{p}(F_{x},\mathbb{Q})$ for any $x\in X_{\Gamma}$. This along with $(ii)$ in Proposition $(\ref{leray-hirsch})$ ensures that  Leray-Hirsch theorem \cite[\S7, Chapter V]{spanier} is applicable for the $K/H$-bundle with projection $Y_\Gamma\to X_\Gamma$. Thus we obtain the following isomorphism:
\[H^*(Y_\Gamma,\mathbb Q)\cong H^*(X_\Gamma,\mathbb Q)\otimes H^*(K/H,\mathbb{Q}).\]      
The following theorem now is immediate from  Hodge decomposition:
\begin{theorem} \label{hodge-decomposition}
For any $p,q\ge 0$, we have the isomorphism of complex vector spaces: 
\[H^{p,q}(Y_\Gamma)\cong \bigoplus_{r\geq 0} 
H^{p-r,q-r}(X_\Gamma)\otimes H^{r,r}(K/H).\]
\end{theorem}
For the remainder of the section we do not assume that $G_{\mathbb{C}}$ is simply connected.  We now obtain a description of the Picard group of $Y_\Gamma$. First we need the following lemma:
\begin{lemma}\label{simplecoefficients}
For any ring $R$ and any integer $q\geq 0$, the usual $R$-linear action of $\pi_{1}(X_{\Gamma}, [e])$ on $H^{q}(F, R)$ is trivial, where $F:=\pi^{-1}([e])$, $\pi$ is the projection $Y_{\Gamma}\rightarrow X_{\Gamma}$ and $e$ is the identity of $G$.
\end{lemma}
\begin{proof}
For any element $\gamma \in \Gamma$, choose a path $\alpha_{\gamma}$ in $G$ from $e$ to $\gamma$, where $e$ is the identity of $G$. Then $\alpha_{\gamma}$ gives a loop in $X_{\Gamma}$ based at $[e]$ which we denote by $\widetilde{\alpha_{\gamma}}$. One readily checks that the map $\gamma\mapsto [\widetilde{\alpha_{\gamma}}]$ is the usual identification of $\Gamma$ with $\pi_{1}(X_{\Gamma},[e])$ as groups. 

Now fix an element $\gamma\in\Gamma$ and an integer $q\geq 0$. The map $\epsilon: K\times I\rightarrow Y_{\Gamma}$ given by $\epsilon(k,t)=[\alpha_{\gamma}(t)k]$ for $k\in K$ and $t\in I$, descends to a map $K/H\times I\rightarrow Y_{\Gamma}$ which under the natural identification of $K/H$ with $F$ gives a map $\tilde{\epsilon}:F\times I\rightarrow Y_{\Gamma}$. Let $j_{0}:F\hookrightarrow F\times I$ be the inclusion given by $y\mapsto (y,0)$. One now checks that $\tilde{\epsilon}$ satisfies the following diagram:

\begin{center}
\begin{tikzpicture}
\matrix(m)[matrix of math nodes,
row sep=2em, column sep=5em,
text height=1.5ex, text depth=0.25ex]
{
F & Y_{\Gamma} \\
F\times I & X_{\Gamma} \\
};
\path[->]
(m-1-1) edge node[above] {$j$} (m-1-2);
\path[->]
(m-1-1) edge node[right] {$j_{0}$} (m-2-1);
\path[->]
(m-2-1) edge node[below] {$(y,t)\mapsto \widetilde{\alpha_{\gamma}}(t)$} (m-2-2);
\path[->]
(m-1-2) edge node[right] {$\pi$} (m-2-2);
\path[->]
(m-2-1) edge node[above] {$\tilde{\epsilon}$} (m-1-2);
\end{tikzpicture}
\end{center}

Finally note that $\tilde{\epsilon}$ restricts to the identity map: $F\times \{1\}\rightarrow F$ implying the triviality of the action of $\pi_{1}(X_{\Gamma},[e])$ on $H^{q}(F,R)$.  
\end{proof}
We are now ready to prove the following theorem:
\begin{theorem} \label{picard-ygamma}  Suppose that $H^{0,2}(X_\Gamma)=0.$
    Then there is a split short exact sequence:
    \[ 1\longrightarrow Pic(X_{\Gamma})\overset{\pi^{!}}\longrightarrow Pic(Y_{\Gamma})\overset{j^{!}}\longrightarrow Pic(F)\longrightarrow 1\]
    where $F=K/H\cong K_\mathbb C/Q$ and $j$ denotes the fiber inclusion $F\hookrightarrow Y_{\Gamma}$. Consequently 
    $Pic(Y_\Gamma)\cong Pic(X_\Gamma)\oplus Pic(F)
    \cong Pic(X_\Gamma)\oplus \mathbb Z^r$ where $r=\dim S/Z(K)$.
\end{theorem}
\begin{proof} One has the following exact sequence for any connected complex manifold $M$:
\begin{equation}\label{exp-exact-seq}
\cdots\rightarrow H^1(M,\mathbb Z)\rightarrow H^{0,1}(M)\rightarrow Pic(M)\overset{c_1}{\rightarrow }H^2(M,\mathbb Z)\rightarrow H^{0,2}(M)\rightarrow\cdots 
\end{equation}
arising from the exponential exact sequence $0\rightarrow \mathbb Z\rightarrow \mathcal O_{M}\rightarrow \mathcal O_{M}^*\to 1$ of sheaves.

As the projection $\pi:Y_\Gamma\to X_\Gamma$ induces isomorphism 
in fundamental groups, we obtain that $\pi^*:H^1(X_\Gamma,\mathbb Z)\to H^1(Y_\Gamma,\mathbb Z)$ is an isomorphism. Also $H^{0,1}(Y_\Gamma)\cong H^{0,1}(X_\Gamma)$. The commutativity of the following diagram:
\begin{center}
\begin{tikzpicture}
\matrix(m)[matrix of math nodes,
row sep=2em, column sep=2.5em,
text height=1.5ex, text depth=0.25ex]
{
H^{1}(X_{\Gamma},\mathbb{Z}) & H^{0,1}(X_{\Gamma}) \\
H^{1}(Y_{\Gamma},\mathbb{Z}) & H^{0,1}(Y_{\Gamma}) \\
};
\path[->]
(m-1-1) edge node[above] {} (m-1-2);
\path[->]
(m-1-1) edge node[right] {$\pi^{*}$} (m-2-1);
\path[->]
(m-2-1) edge node[below] {} (m-2-2);
\path[->]
(m-1-2) edge node[right] {$\pi^{*}$} (m-2-2);
\end{tikzpicture}
\end{center}
implies that $\pi^{!}:Pic(X_{\Gamma})\rightarrow Pic(Y_{\Gamma})$ restricts to an isomorphism $Pic^0(X_\Gamma)\cong Pic^0(Y_\Gamma)$ where $Pic^0(Y_\Gamma)$ is the kernel of $c_1:Pic(Y_\Gamma)\to H^2(Y_\Gamma,\mathbb Z)$. 

As already observed in Remark $(\ref{picard-surj})$, 
$j^{!}:Pic(Y_{\Gamma})\rightarrow Pic(F)$ and $j^{*}:H^{2}(Y_{\Gamma},\mathbb{Z})\rightarrow H^{2}(F,\mathbb{Z})$ are surjections.   
In view of Lemma $(\ref{simplecoefficients})$, the local coefficient system $\mathcal H^q(F,\mathbb Z)$ in the 
Leray-Serre spectral sequence of the bundle $F\overset{j}\rightarrow Y_{\Gamma}\overset{\pi}\rightarrow X_{\Gamma}$ is simple.  Now the vanishing of 
$H^1(F,\mathbb Z)$ and the surjectivity of 
$j^{*}:H^{2}(Y_{\Gamma},\mathbb{Z})\rightarrow H^{2}(F,\mathbb{Z})$  yields the 
exactness of the following sequence:
\[0 \rightarrow H^{2}(X_{\Gamma},\mathbb{Z})\overset{\pi^{*}}\rightarrow H^{2}(Y_{\Gamma},\mathbb{Z})\overset{j^{*}}\rightarrow H^{2}(F,\mathbb{Z})\rightarrow 0.\]
Finally, the long exact sequence $(\ref{exp-exact-seq})$ 
gives the exact sequence: $1\rightarrow Pic^0(X_\Gamma)\rightarrow Pic(X_\Gamma)\overset{c_{1}}{\rightarrow} H^2(X_\Gamma, \mathbb{Z})\rightarrow  0$
since, by hypothesis $H^{0,2}(X_{\Gamma})=0$.
From the isomorphism $Pic^{0}(X_{\Gamma})\cong Pic^{0}(Y_{\Gamma})$, 
as pointed out earlier in the proof and the fact $Pic(F)\cong H^{2}(F,\mathbb{Z})\cong \mathbb{Z}^{r}$, the theorem follows. 
\end{proof}
The following remark is an application of the Borel spectral sequence \cite[Appendix (II), \S 2]{hirzebruch} to the holomorphic bundle $\xi_{\Gamma}$ with projection $\pi:Y_{\Gamma}\rightarrow X_{\Gamma}$, fiber $K/H$ and holomorphic structure group $P$.
\begin{remark}\label{Borelspectralsequence}
We keep the above notations. Then the following are true:
$(i)$ For any $q\geq 1$, $\pi^{*}:H^{0,q}(X_{\Gamma})\rightarrow H^{0,q}(Y_{\Gamma})$ is an isomorphism; 
$(ii)$ There is an exact sequence $0\rightarrow H^{1,1}(X_{\Gamma})\overset{\pi^{*}}\rightarrow H^{1,1}(Y_{\Gamma})\overset{\iota^{*}}\rightarrow H^{1,1}(F)$; 
$(iii)$ If $q\geq 2$ and $H^{0,q-1}(X_{\Gamma})=0$ then $\pi^{*}:H^{1,q}(X_{\Gamma})\rightarrow H^{1,q}(Y_{\Gamma})$ is surjective and 
$(iv)$ If $q\geq 3$ and $H^{0,q-2}(X_{\Gamma})=0$ then $\pi^{*}:H^{1,q}(X_{\Gamma})\rightarrow H^{1,q}(Y_{\Gamma})$ is injective.
\end{remark}
We end this section with the following theorem which considers $Pic(X_{\Gamma})$.
Recall that the lattice $\Gamma\subset G$ is {\em irreducible} if there does not exist a finite cover of $X_\Gamma$ which is a product of two compact locally Hermitian symmetric spaces of positive dimensions. If $\Gamma$ is an irreducible lattice, then $r_\mathbb R(G)\ge 2$ implies  $[\Gamma,\Gamma]\subseteq \Gamma$ is a finite index subgroup whereas 
$r_{\mathbb{R}}(G)\geq 3$ implies $H^{0,1}(X_\Gamma)=0=H^{0,2}(X_\Gamma)$ and $H^{2}(X_\Gamma,\mathbb C)=H^{1,1}(X_\Gamma)\cong H^{1,1}(X_{\textrm{u}})$, see \cite[Ch. VII, \S 4, Corollary $(4.4)$ $(b)$]{borel-wallach}, where $X_{\textrm{u}}$ denotes the compact dual of $X$. 
\begin{theorem}\label{picard-xgamma}
Suppose that $\Gamma\subset G$ is an irreducible lattice and that $r_\mathbb R(G)\ge 3$. 
Let $n$ be the number of simple factors of $G$. 
Then  
$H^{2}(X_{\Gamma},\mathbb{Z})\cong \mathbb Z^n\oplus \Gamma/[\Gamma,\Gamma]$.   
Moreover, the first Chern class map $c_{1}:Pic(X_{\Gamma})\rightarrow H^{2}(X_{\Gamma},\mathbb{Z})$ is an isomorphism.  
\end{theorem}
\begin{proof}
As noted above, our hypotheses on $r_\mathbb R(G)$ and the irreducibility of $\Gamma$ implies the vanishing of $H^{0,1}(X_\Gamma)$ and $H^{0,2}(X_\Gamma).$  Therefore $H_1(X_\Gamma,\mathbb Z)\cong \Gamma/[\Gamma,\Gamma]$ is finite.  Consequently, $H^1(X_\Gamma,\mathbb Z)=0.$
Since $G$ has $n$ simple factors (and has no compact factors by our blanket assumption on $G$), 
$X$ and its compact dual $X_{\textrm{u}}$, which is also a generalized complex flag manifold,  are a product of $n$ {\it irreducible} 
Hermitian symmetric spaces. Therefore $H^2(X_{\textrm{u}},\mathbb Z)\cong \mathbb Z^n.$  
It follows that $H^2(X_\Gamma,\mathbb C)=H^{1,1}(X_\Gamma)\cong 
H^{1,1}(X_{\textrm{u}})=H^{2}(X_{\textrm{u}},\mathbb{C})\cong 
\mathbb C^n$. 
Now a simple application of the universal coefficient theorem for cohomology shows that $H^{2}(X_{\Gamma},\mathbb{Z})\cong \mathbb{Z}^{n}\oplus \Gamma/[\Gamma,\Gamma]$.
The last assertion follows from the long exact sequence in cohomology associated with the 
exponential short exact sequence of sheaves on $X_{\Gamma}$.
\end{proof}
\section{The Hodge numbers $h^{p,q}(X_\Gamma)$ when $p\le 1$}
 We keep the notations, $T\subseteq K\subset G, X=G/K,$ and hypotheses of the previous section.  In addition, 
 in this section we assume that $G$ is simple and $r_\mathbb R (G)\ge 2$ unless otherwise stated. Note that 
as $G$ is simple, any lattice in $G$ is irreducible. 
 Recall that any maximal torus of $K$ contains $Z(K)$, the centre of $K$. 

A positive root system $\Delta^+\subset \Delta$ of $(\mathfrak g,\mathfrak t)$ is fixed 
as in \S 2, $\Delta_K$ (resp. $\Delta_{\textrm{n}}$) denotes the set of compact (resp. non-compact) roots, $\Delta^+_K=\Delta^+\cap \Delta_K$ and $ \Delta^+_{\textrm{n}}=\Delta^+\cap \Delta_{\textrm{n}}$.   We shall denote by $\Sigma\subset \Delta^+$
the set of simple roots.
As $G$ is simple, there is exactly one non-compact simple root.

It is well-known \cite[\S4, Chapter VII]{borel-wallach} that, 
$H^{p,q}(X_\Gamma) =0$ if $p\neq q$ and $p+q< r_\mathbb R(G)$ for any irreducible, torsion free and uniform
lattice $\Gamma$ in $G$.  
Parthasarathy in \cite[\S6]{parthasarathy-1980} has shown the vanishing of $H^{0,q}(X_\Gamma)$ for values of $q$ that exceed $r_{\mathbb{R}}(G)$, depending on the 
Cartan type of $X$. As only the `results of doing the exercise' are given in \cite{parthasarathy-1980}, we include the proofs here.   
Parthasarathy's approach was to show the non-existence of $\theta$-stable parabolic subalgebra of $\mathfrak{g}_{0}$ with prescribed $R^-(\mathfrak q)$ when $R^+(\mathfrak q)=0$.  We follow his approach and do a thorough analysis in all the Cartan types, also obtaining, in the same spirit, the vanishing results for 
$H^{1,q}(X_\Gamma).$

Let $x\in \mathfrak t_\mathbb R$.  Denote by $\mathfrak q=\mathfrak q_x\subseteq \mathfrak g$ the parabolic subalgebra of $\mathfrak{g}$ defined as: 
\[ \mathfrak q_x=\mathfrak t\oplus(\bigoplus_{\alpha\in \Delta, \alpha(x)\ge 0} \mathfrak g_\alpha).\]
Then $\mathfrak q_x$ is a $\theta$-stable parabolic subalgebra of $\mathfrak{g}_{0}$.   
If $x=0$, then $\mathfrak q_x=\mathfrak g$ and we shall leave out this case from 
further discussion.  Let $x\ne 0.$  Then $\mathfrak q_x$ is a proper parabolic subalgebra of $\mathfrak g$.

It is readily seen that 
$\mathfrak q_x=\mathfrak l_x\oplus \mathfrak u_x$ where $\mathfrak l_x=\mathfrak t\oplus (\bigoplus_{\alpha(x)=0}\mathfrak g_\alpha)\subseteq \mathfrak g$ equals the Levi 
subalgebra of $\mathfrak q_x$ and $\mathfrak u_x=\bigoplus_{\alpha(x)>0}\mathfrak g_\alpha$ is the nil radical of $\mathfrak q_{x}$.  Denote by $L_x\subseteq G$ the connected 
Lie subgroup of $G$ with Lie algebra $\mathfrak l_{x,0}=\mathfrak l_x\cap\overline{\mathfrak l_x}\subseteq \mathfrak g_0.$  
It turns out that any 
$\theta$-stable parabolic subalgebra of $\mathfrak{g}_{0}$ that contains $\mathfrak t$ arises as $\mathfrak q_x$ for some $x\in \mathfrak t_\mathbb R.$  

Let $k\in N_K(T)$, the normalizer of $T$ in $K$, and let $y=Ad(k)(x)\in \mathfrak t_\mathbb R$ for $x\in\mathfrak{t}_{\mathbb{R}}$.  Then $Ad(k)\mathfrak q_x=\mathfrak q_y=\mathfrak l_y\oplus \mathfrak u_y$
where $\mathfrak l_y=Ad(k)\mathfrak l_x, \mathfrak u_y=Ad(k)\mathfrak u_x.$
Note that since $\mathfrak p^\pm$ is a $K$-representation, $Ad(k)(\mathfrak p^\pm)=\mathfrak p^\pm$ and we have $\mathfrak u_y\cap\mathfrak p^\pm\cong u_x\cap \mathfrak p^\pm.$ Define $R^\pm(\mathfrak q_x):=\dim_\mathbb C(\mathfrak u_x\cap\mathfrak p^\pm)$.  Then $R^\pm(\mathfrak q_x)=R^\pm(\mathfrak q_y).$

Define $\mathcal R^+=\mathcal R^+(\mathfrak q_x)\subseteq \Delta_{\textrm{n}}^+$ as follows:  $\mathcal R^+:=\{\alpha\in\Delta_{\textrm{n}}^+\mid \alpha(x)>0\}$. $\mathcal R^-$ is defined analogously.
Thus $|\mathcal R^\pm|=R^\pm(\mathfrak q_x)=
\dim \mathfrak u_x\cap \mathfrak p^\pm$. It is well known from the works of R. Parthasarathy, D. Vogan and G. Zuckerman (see \cite{parthasarathy-1978}, \cite{vogan} and \cite{vogan-zuckerman}),  that for any $\theta$-stable parabolic subalgebra $\mathfrak{q}$ of $\mathfrak{g}_{0}$ containing $\mathfrak{t}$ there is an irreducible unitary representation $\mathcal{A}_{\mathfrak{q}}$ of $G$ (whose space of smooth and $K$-finite vectors would be denoted by $A_{\mathfrak{q}}$) and conversely for any irreducible unitary representation $V_{\pi}$ of $G$ with $H^{k}(\mathfrak{g}_{0},K,V_{\pi,K}^{\infty})\neq 0$ for some $k$, there exists a $\theta$-stable parabolic subalgebra $\mathfrak{q}$ of $\mathfrak{g}_{0}$ containing $\mathfrak{t}$ such that $V_{\pi}\cong \mathcal{A}_{\mathfrak{q}}$ as irreducible unitary representations of $G$. Furthermore up to unitary equivalence $\mathcal{A}_{\mathfrak{q}}$ depends only on the $K$-conjugacy class of $\mathfrak{q}$.

We shall write $\mathfrak q, \mathfrak l$, $L$ and $\mathfrak u$ for $\mathfrak q_x, \mathfrak l_x$, $L_x$ and $\mathfrak u_x$ respectively for $x\in\mathfrak{t}_{\mathbb{R}}$. The space $L/L\cap K$ is Hermitian symmetric
and we have the following isomorphism for all $p,q\geq 0$:
\begin{equation}\label{gkcohomology-hodgetype}
 H^{p,q}(\mathfrak g_0,K,A_{\mathfrak q})\cong H^{p-R^+({\mathfrak{q})},q-R^-({\mathfrak{q}})}(\mathfrak l_0, K\cap L,
 \mathbb C).
\end{equation}
Denote by $L_{\textrm{u}}$ the maximal compact subgroup of $L_\textrm{c}\subseteq G_\mathbb C$. 
Here $L_\textrm{c}$ denotes the connected Lie subgroup of $G_\mathbb C$ corresponding to the Lie subalgebra 
$\mathfrak l\subset \mathfrak g.$
Thus $Y_{\mathfrak q}:=L_{\textrm{u}}/L_{\textrm{u}}\cap K$ is the compact dual of $L/L\cap K$.   As $L/L\cap K$ is a Hermitian symmetric domain, so is $Y_{\mathfrak q}$ and $Y_{\mathfrak q}$ is also a generalized complex flag manifold.
Since for any $p,q\geq 0$, $H^{p,q}(\mathfrak l_0, L\cap K,\mathbb C)\cong H^{p,q}(Y_{\mathfrak q})$, the isomorphism $(\ref{gkcohomology-hodgetype})$ yields the following for any $p,q\geq 0$:
\begin{equation}\label{gk-cohomology-hermitiansymmetricspace}
    H^{p,q}(\mathfrak g_0,K,A_{\mathfrak q})
    \cong  H^{p-R^+({\mathfrak{q}}),q-R^-({\mathfrak{q}})}(Y_{\mathfrak q}).
\end{equation}

Note that if there is no $\theta$-stable parabolic subalgebra $\mathfrak{q}$ of $\mathfrak{g}_{0}$ containing $\mathfrak{t}$ such that $R^{+}(\mathfrak{q})=1=R^{-}(\mathfrak{q})$, then there cannot be a non-trivial irreducible and unitary representation $V_{\pi}$ of $G$ with $H^{1,1}(\mathfrak{g}_{0},\mathfrak{k}_{0},V^{\infty}_{\pi,K})\neq 0$. Indeed if $H^{1,1}(\mathfrak{g}_{0},\mathfrak{k}_{0},V^{\infty}_{\pi,K})\neq 0$, then by the above discussions, $V^{\infty}_{\pi,K}=A_\mathfrak{q}$ for some $\theta$-stable parabolic subalgebra $\mathfrak{q}$ of $\mathfrak{g}_{0}$ containing $\mathfrak{t}$. Thus by $(\ref{gk-cohomology-hermitiansymmetricspace})$, $H^{1,1}(\mathfrak{g}_{0},\mathfrak{k}_{0},V^{\infty}_{\pi,K})=H^{1,1}(\mathfrak{g}_{0},\mathfrak{k}_{0},A_\mathfrak{q})=H^{1-R^{+}(\mathfrak{q}),1-R^{-}(\mathfrak{q})}(Y_{\mathfrak{q}})$. Therefore it is forced that $R^{+}(\mathfrak{q})=0=R^{-}(\mathfrak{q})$ as $Y_{\mathfrak{q}}$ is also a generalized complex flag manifold. Hence again by $(\ref{gk-cohomology-hermitiansymmetricspace})$, $H^{0,0}(\mathfrak{g}_{0},\mathfrak{k}_{0},V^{\infty}_{\pi,K})=H^{0,0}(\mathfrak{g}_{0},\mathfrak{k}_{0},A_\mathfrak{q})=H^{0,0}(Y_{\mathfrak{q}})\neq 0$. Equation $(\ref{hodgematsushima})$ now shows that $H^{0,0}(X_{\Gamma})$ has complex dimension at least $2$ which is a contradiction as $X_{\Gamma}$ is connected. Hence as $G$ is simple, $(\ref{hodgematsushima})$ again forces $H^{1,1}(X_{\Gamma})\cong \mathbb{C}$.

Therefore we conclude the following remark from the above discussions and the equations  $(\ref{hodgematsushima})$ and $(\ref{gk-cohomology-hermitiansymmetricspace})$:
\begin{remark}\label{vanishingcohomology}
$(i)$ Let $p$ be any positive integer. If there is no $\theta$-stable parabolic subalgebra $\mathfrak{q}$ containing $\mathfrak{t}$ with 
$(R^{+}(\mathfrak{q}),R^{-}
(\mathfrak{q}))=(0,p)$, then $H^{p,0}(X_{\Gamma})=0=H^{0,p}(X_{\Gamma})$.

$(ii)$ Let $p\geq 2$ be an integer. If for any $\theta$-stable parabolic subalgebra $\mathfrak{q}$ containing $\mathfrak{t}$ we have  
$(R^{+}(\mathfrak{q}),R^{-}(\mathfrak{q}))\notin \{(1,p),(0,p-1)\}$  
then $H^{1,p}(X_{\Gamma})=0=H^{p,1}(X_{\Gamma})$.

$(iii)$ If there is no $\theta$-stable parabolic subalgebra $\mathfrak{q}$ containing $\mathfrak{t}$  with $(R^{+}(\mathfrak{q}),R^{-}(\mathfrak{q}))=(1,1)$, then $H^{1,1}(X_{\Gamma})\cong \mathbb{C}$.
\end{remark}

Note that parts $(i)$ and $(ii)$ in the above remark do not require $G$ to be simple, whereas it is essential for part $(iii)$.

Our strategy of proof would be to exploit the above remark. Since 
$R^\pm (\mathfrak q_x)=R^\pm(\mathfrak q_y)$ if $x,y\in \mathfrak t_\mathbb R$ are in the same orbit of the Weyl group $W_K=W(K,T)$, we need only verify our assertion for $x$ in a  conveniently chosen (closed) $W_K$-chamber of $\mathfrak t_\mathbb R.$ 
The proof will be by case consideration depending on the simple group $G$, which are 
tabulated below for the convenience of the reader. (See \cite[\S2, Chapter X]{helgason}). We shall follow the  description of the root system and 
the labeling of simple roots as in Bourbaki \cite{bourbaki}. 

We shall specify in each case which simple root is non-compact,
the set of positive compact and non compact roots and also the set of simple roots.  We will also 
specify the Weyl group of $K$ and describe its action on $\mathfrak t_\mathbb R.$
Both $\mathfrak t_\mathbb R$ and its 
dual $\mathfrak t_\mathbb R^*$, of which $\Sigma$ is a basis, will be described as subspaces of the Euclidean space $\mathbb R^N$.  
\[
\begin{array}{|c|c|c|c|c|c|}
  \hline
  Type   & G ~(or~ \mathfrak g_0) & K & G_\mathbb C & r_\mathbb R(G) & dim_{\mathbb{C}}(X) \\
  \hline\hline
   {\bf AIII}  & SU(m,n)&S(U(m)\times U(n)) & SL(m+n,\mathbb C) & \min\{m,n\} & mn\\
   \hline
   {\bf BDI} & SO_0(2, q) (q\geq 2)& SO(2)\times SO(q) & SO(2+q,\mathbb{C})& 2 & q\\
   \hline
   {\bf CI} & Sp(n,\mathbb R) & U(n) & Sp(n,\mathbb{C}) & n & n(n+1)/2\\
   \hline
   {\bf DIII}& SO^*(2n) & U(n)& SO(2n,\mathbb C) & [\frac{n}{2}] & n(n-1)/2\\
   \hline
   {\bf EIII} & \mathfrak e_{(6,-14)} & Spin(10)\cdot U(1) & E_{6,\mathbb C} & 2 & 16\\
   \hline
   {\bf EVII}& \mathfrak e_{(7,-25)} & E_6\cdot U(1) &E_{7,\mathbb C}& 3 & 27\\
   \hline
\end{array}
\]
\begin{center}
    {\bf Table 1:} Irreducible Hermitian symmetric pairs $(G,K)$ of non-compact type.
\end{center}
\subsection{Type AIII} We will first consider the case when the $r_{\mathbb{R}}(G)$ is at least $2$.  The case when $r_{\mathbb{R}}(G)=1$, although easier, will be treated separately.
Let $2\le m\le n$. $G=SU(m,n), K=S(U(m)\times U(n)).$  \\
$\Sigma=\{\psi_j:=e_j-e_{j+1}\mid 1\le j<m+n\}$; $\psi_m=e_m-e_{m+1}$ is the non-compact simple root. 
$\Delta^+_K=\{e_i-e_j\mid 1\le i<j\le m\}\cup\{e_i-e_j\mid m<i<j\le m+n\},$\\
$\Delta^+_{\textrm{n}}=\{e_i-e_j\mid 1\le i\le m<j\le m+n\}.$\\
$\mathfrak t_\mathbb R\subset \mathbb R^{m+n}$ defined by the equation 
$\sum_{1\le j\le m+n} x_j=0.$ \\
$W_K= S_m\times S_n\subset S_{m+n}$ acts on $\mathfrak t_\mathbb R$ by permuting 
the first $m$ coordinates and the last $n$ coordinates.  

Let $x\in \mathfrak t_\mathbb R\subset \mathbb R^{m+n}, x\ne 0.$  We may (and do) assume that 
$x_1\le \cdots \le x_m$ and $x_{m+1}\le  \cdots \le x_{m+n}.$

Suppose that $R^+=R^+(\mathfrak q_x)=0.$ Then $x_i\le x_j$ if $1\le i\le m<j\le m+n$ and so $x_1\le \cdots\le x_m\le x_{m+1}\le \cdots\le x_{m+n}.$
Now $-(e_i-e_j)\in\mathcal R^-$  if and only if $i\le m<j\le m+n$ and $x_i<x_j.$ 
If $x_1=x_{m+n},$ then $\sum_{1\le j\le m+n} x_j=0$ implies that $x=0$, contrary to our choice.  
If $x_1=x_m$ we set $r=0$ otherwise 
let $r\le m-1$ be the largest positive integer such that $x_r<x_m$.  Similarly, if $x_m=x_{m+n},$
we set $t=0$; otherwise let $t$ be the largest number such that $1\le t\le n$ and 
$x_m<x_{m+n+1-t}.$ Note that both $r$ and $t$ cannot vanish.  
Then $e_j-e_i\in \mathcal R^-$ if and only if  $m+n+1-t\le j\le m+n$ or $1\le i\le r.$ 
Therefore $R^-=rn+(m-r)t.$ 

Conversely, if $1\le r< m$ and $1\le t\le n$, then $(R^+,R^-)=(0,rn+(m-r)t)$ is 
realized as $(R^+(\mathfrak q_a),R^-(\mathfrak q_a))$ where $$a=(\underbrace{-t,\cdots, -t}_{\textrm{$r$ terms}}, \underbrace{0,\cdots, 0}_{\textrm{$m+n-t-r$ terms}}, \underbrace{r,\cdots,r}_{\textrm{$t$ terms}}).$$

Also $(R^+(\mathfrak q_b),R^-(\mathfrak q_b))=(0, mt), (R^+(\mathfrak q_c),R^-(\mathfrak q_c))=(0, nr)$ when 
$$b=(\underbrace{-t,\cdots, -t}_{\textrm{$m+n-t$ terms}}, \underbrace{m+n-t,\cdots,m+n-t}_{\textrm{$t$ terms}})$$ and 
$$c=(\underbrace{-(m+n-r),\cdots, -(m+n-r)}_{\textrm{$r$ terms}},  \underbrace{r,\cdots,r}_{\textrm{$m+n-r$ terms}})$$
respectively.

Next assume that $R^+=R^+(\mathfrak q_x)=1$, where $x_1\le \cdots\le x_m, x_{m+1}\le 
\cdots \le x_{m+n}$. Then $\mathcal R^+$ equals $\{e_m-e_{m+1}\}$ and    
we must have $x_{m-1}\le x_{m+1}<x_{m}\le x_{m+2}$. Thus we have  
\[x_1\le \cdots\le x_{m-1}\le x_{m+1}<x_m \le x_{m+2}\le \cdots\le x_{m+n}.\]
Now $x_i-x_j<0$ for $1\le i<m, m+2\le j\le m+n$ and so $e_j-e_i\in \mathcal R^-$.  
If $x_{1}=x_{m+1}$ set $r=0$; otherwise set $r$ to be the largest integer in the range $1\leq r\leq m-1$ such that $x_{r} < x_{m+1}$. Similarly if $x_{m}=x_{m+n}$ set $s=m+1$ otherwise set $s$ to be the smallest integer in the range $m+2\leq s\leq m+n$ such that $x_{m} < x_{s}$. From the above observations it follows that $R^{-}(\mathfrak{q}_{x})=(m-1)(n-1)+r+(m+n-s+1)=mn+r-s+2$. 
Moreover, given 
$r,s$ with $0\leq r\leq m-1, m+1\leq s\le m+n,$ we can find a suitable $b=(b_j)\in \mathfrak t_\mathbb R$ so that  
$R^+(\mathfrak q_b)=1, R^-(\mathfrak q_b)=mn+r-s+2.$ For example, we may take 
$b$ to be as follows:
\[b_j=\begin{cases}
    -2(m+n-s+1),& ~\textrm{if~} 1\le j\le r,\\
    -(m+n-s+1),&~\textrm{if~} r+1\le j\le m+1, j\ne m,\\
     0,& ~\textrm{if~} j=m,m+2,m+3,...,s-1,\\
     m+r,& ~\textrm{if~} s\le j\le m+n.\\
\end{cases}\]

Now we turn to the rank $1$ case, namely, $G=SU(1,n).$  We assume that 
$n>1$ since, when $G=SU(1,1)$, $X_\Gamma$ is a compact Riemann surface, which 
is well understood.

We have $\Sigma=\{e_i-e_{i+1}\mid 1\le i\le n\}$ with $\psi_1=e_1-e_2$ being the non-compact simple root. 
Also, 
$\Delta_K^+=\{e_i-e_j\mid 2\le i<j\le n+1\}, \Delta_{\textrm{n}}^+=\{e_1-e_j\mid 1<j\le n+1\}$ and $W_K\cong S_n$ which acts on $\mathfrak t_\mathbb R$ by 
permuting the coordinates $x_2,\ldots, x_{n+1}$ of $x=(x_1,\ldots, x_{n+1})\in \mathfrak t_\mathbb R.$ As in the higher rank case we have $\sum_{1\le j\le n+1}x_j=0$.   Let $x\in \mathfrak t_\mathbb R$ be non-zero. We assume that $x_2\le \cdots\le x_{n+1}.$

Suppose that $R^+(\mathfrak q_x)=0.$   Then $x_1\le x_j$ for $2\le j\le n+1.$ 
If equality holds for all $2\le j\le n+1$, then $0=\sum_{1\le j\le n+1}x_j=(n+1)x_1$, which implies that $x=0,$ contrary to our hypothesis. 
Suppose that $s$ is the least number such that $2\le s\le n+1$ and $x_1<x_s.$  Then $-(e_1-e_j)\in \mathcal R^-(\mathfrak q_x)$ for $s\le j\le n+1$ and so,
$R^-(\mathfrak q_x)=n+2-s.$
It is easily seen that, given any $s$, $2\le s\le n+1$, there exists an $x\neq 0\in 
\mathfrak t_\mathbb R$ such that $R^+(\mathfrak q_x)=0, R^-(\mathfrak q_x)=n+2-s.$
{\em Therefore, when $R^+=0$,  we have $1\le R^-\le n$.}

Suppose that $R^+(\mathfrak q_x)=1.$  Then $\mathcal R^+(\mathfrak q_x)=\{e_1-e_{2}\}$ and $x_2<x_1\le x_j~\forall j\ge 3.$  If $x_1=x_j~\forall j\ge 3$, then $R^-=0.$   Otherwise, there exists a least integer $s\ge 3$ such that $x_1<x_s.$
Now $-(e_1-e_j)\in \mathcal R^-(\mathfrak q_x)$ for $s\le j\le n+1$ and so  
$R^-(\mathfrak q_x)=n+2-s.$   It is 
easily seen that any integer between $0$ and $n-1$ is realized as $R^-(\mathfrak q_x)$ with $R^+(\mathfrak q_x)=1$ for some $x\neq 0\in\mathfrak{t}_{\mathbb{R}}$. {\em Thus, when $R^+=1,$ we have $0\le R^-\le n-1.$ }
\subsection{Type BDI(rank=2) (rank $K$ is even)} Let $G=SO_0(2, 2m-2),m\ge 4.$  \\
$\Sigma=\{ \psi_j=e_j-e_{j+1}\mid 1\le j<m\}\cup\{\psi_m=e_{m-1}+e_m\}$ with non-compact simple root being $\psi_1=e_1-e_2,$  \\
$\Delta_K^+=\{(e_i-e_j), (e_i+e_j)\mid 2\le j\le m\}$,\\
$\Delta_{\textrm{n}}^+=\{e_1\pm e_j\mid 2\le j\le m\}$.
$W_K\cong S_{m-1}\rtimes A$ where $A\subset \{1,-1\}^{m-1}$ where $a=(a_j)\in A $ if $\prod a_j=1.$  The group $W_K$ acts on $\mathfrak t_\mathbb R\cong \mathbb R^m$ by permuting the coordinates $x_2,\ldots, x_m$ and changing the signs of even number of them where 
$x=(x_1,\ldots, x_m)\in \mathfrak t_\mathbb R$.

Assume $m\geq 3$. Let $x=(x_1,\ldots, x_m)\in \mathfrak t_\mathbb R$, $x\ne 0$  and $\mathfrak q=\mathfrak q_x$.  
The Weyl group $W_K=W(K,T)$ acts on $\mathfrak t_\mathbb R$ by permuting the coordinates $x_2,\ldots, x_m$ and changing the signs of even number of them.  So we may (and do) assume that $x_2\ge \cdots \ge x_{m-1}\ge|x_m|.$ 
Since $x\ne 0$, we have $x_1\ne 0$ or $x_2>0$.
Suppose that $R^+=0.$  Then $x_1+x_j\le 0, x_1-x_j\le 0$ for all $j\ge 2$. So, $x_1\le -x_2.$  If $x_1<-x_2,$ then $\mathcal R^-=\Delta^-_{\textrm{n}}$ and we have $R^-=2m-2$.

Suppose that $x_1=-x_2.$   Then $x_1<0$ and so $x_1-x_j<0$ for $2\le j<m,$ and $x_1-|x_m|<0$.
Therefore $-(e_1-e_j)\in \mathcal R^-$ for $2\le j<m$ and $-(e_1-\varepsilon e_m)\in \mathcal R^-$ where $\varepsilon\in \{1,-1\}$ is such that $\varepsilon x_m\ge 0.$ Also, if $x_1<-x_j$, then $-(e_1+e_j)\in \mathcal R^-.$  Let $J=\{3\le j\le m\mid x_1<-|x_j|\}.$ Then $0\le |J|\le m-2$ and  $-(e_1\pm e_j)\in \mathcal R^-$ for $j\in J$. 

{\em We conclude that, if $R^+=0$, then $R^-\in\{m-1, m, \cdots, 2m-2\}.$}   
All these values of $R^{-}$ are realized for suitable choices of non-zero $x\in\mathfrak{t}_{\mathbb{R}}$.

Now assume that $m\geq 4$. Suppose that $R^+(\mathfrak q)=1.$ Then, as $x_1+x_2\ge x_1\pm x_j~\forall j>1$, we must have 
$x_1+x_2>0, $ and it follows that $x_1-x_2< 0$,  $x_1\pm x_j\le 0~\forall j>2$.  In particular, $x_2>x_3\ge 0.$  
Suppose that $x_1=0$. Then $x_j=0~\forall j\ge 3$  (as otherwise $x_1+x_j>0$ for some $j>2$ which would imply that $R^+\ge 2$).  It follows that $\mathcal R^-=\{e_2-e_1\}$ and we have $R^+=R^-=1$.  

Suppose that $x_1<0.$ Then $x_1-|x_j|<0~\forall j\ge 2$ and so $-(e_1-e_j)\in \mathcal R^-$ for $2\le j<m$ and $-(e_1-\varepsilon e_m)\in \mathcal R^-$  where $\varepsilon\in\{1,-1\}$ is such that $\varepsilon x_m\ge 0.$
Let $J=\{j\ge 3\mid x_1<-|x_j|\}.$  Then $-(e_1\pm e_j)\in \mathcal R^-$ for $j\in J, j\ne m$.
So $R^-=(m-1)+c$ where 
$c=|J|\le m-2.$ Any value of $c\le m-2$ can be realized by a suitable choice of non-zero $x\in\mathfrak{t}_{\mathbb{R}}$.
{\it Thus, if $R^+=1,$ then $R^-$ takes the following values: $1, m-1+c, 0\le c\le m-2.$}
\subsection{Type BDI(rank=2), (rank $K$ is odd)} Let $G=SO_0(2,2m-1), m\ge 2.$ \\
$\Sigma=\{\psi_i=e_i-e_{i+1}\mid 1\le i\le m-1\}\cup\{\psi_m=e_m\}$, the non-compact simple root being $\psi_1.$ \\
$\Delta_K^+=\{e_i\pm e_j\mid 2\le i<j\le m\}\cup\{e_j\mid 2\le j\le m\}.$\\
$\Delta^+_{\textrm{n}}=\{e_1\pm e_j\mid 2\le j\le m\}\cup\{e_1\}.$\\
$W_K=S_{m-1}\rtimes \mathbb Z_2^{m-1}$ acts on $\mathfrak t_\mathbb R\cong \mathbb R^n$  
by permuting $x_2,\ldots, x_m,$ and changing the $x_j$ to $-x_j$ for $2\le j\le m$, where $x=(x_1,\ldots, x_m)\in \mathfrak t_\mathbb R$. 

Let $x=(x_1,\ldots,x_m)\in \mathfrak t_\mathbb R, x\ne 0.$  
We assume that $x_2\ge  \cdots\ge x_m\ge 0.$

Suppose that $R^+(\mathfrak q_x)=0.$ Then $x_1\le 0$ and $x_1+x_j\le 0, x_1-x_j\le 0$ for $2\le j\le m$.

Suppose that $x_1=0$.  Then $x_2>0$ as otherwise $x=0$, contrary to our assumption that $x\ne 0$.
But then $e_1+e_2\in \mathcal R^+$, a contradiction.  This shows that $x_1<0.$
Now $-(e_1-e_j)\in \mathcal R^-~\forall j\ge 2$ and $-e_1\in \mathcal R^-$.  
Let $c=|\{ 2\le j\le m\mid -x_1=x_j\}|.$  Then $-(e_1+e_k)\in \mathcal R^-$ if $c+1<k\le m.$
Note that $c$ can take any value between $0$ and $m-1$ for suitable choices of non-zero $x\in\mathfrak{t}_{\mathbb{R}}$.   
Then  $R^-=2m-1-c.$  

{\em Thus $R^+=0$ implies that $m\le R^-\le 2m-1$ and any such value is attained for suitable choices of non-zero $x\in\mathfrak{t}_{\mathbb{R}}$}. 

Suppose that $R^+(\mathfrak q_x)=1.$   Suppose that $x_1=0$. Then $x_2\ge x_j~\forall j\ge 2$ 
implies that $\mathcal R^+=\{e_1+e_2\}.$ Also, $x_j=0$ for $j\ge 3$ since otherwise $e_1+e_3\in 
\mathcal R^+.$
Now $-(e_1+e_2)\in \mathcal R^-$, $-(e_1-e_2)\notin \mathcal R^-$ and $-(e_1\pm e_k)\notin \mathcal R^-$ for $3\le k\le m.$  So $R^-=1.$

Suppose that $x_1>0.$ Then $e_1\in \mathcal R^+$.    Since $x_2\ge 0,$ then 
$e_1+e_2\in \mathcal R^+$, forcing $R^+\ge 2$.  So this case cannot occur.

Next, suppose that $x_1<0$.  Then $\mathcal R^+=\{e_1+e_2\}$ and we must have $x_2>x_j$, $x_j\le -x_1 ~\forall j\ge 3$ and $x_2>0$. 
First assume $m\geq 3$. Let $3\le c\le m$ be the smallest number such that $x_1+x_c<0$ if exists.
As we have $-(e_1-e_j)\in \mathcal R^-~\forall ~ 2\le j\le m$ and $-e_1\in \mathcal R^-$, $R^-=1+(m-1)+(m-c+1)=2m-c+1$. If such a $c$ does not exists, then $R^-=1+(m-1)=m$. Moreover for every $c\ge 3$, there is a non-zero $x$ in $\mathfrak{t}_{\mathbb{R}}$ such that $R^-=2m-c+1$, whereas if $m=2$, then $R^-=m$.   
{\em Thus, if $R^+=1$, then $R^-\in \{q\mid q=1\textrm{~or~} m\le q\le 2m-2\}$. All these 
values are realized for suitable choices of non-zero $x\in\mathfrak{t}_{\mathbb{R}}$.}
\subsection{Type CI} Let $G=Sp(n,\mathbb R), K=U(n)\subset G, n\ge 2.$ \\ 
$\Sigma=\{\psi_i=e_i-e_{i+1}\mid 1\le i\le n\}\cup \{\psi_n=2e_n\}$, with non-compact simple root being $\psi_n,$\\
$\Delta_K^+=\{e_i-e_j\mid 1\le i<j\le n\},$\\
$\Delta_{\textrm{n}}^+=\{e_i+e_j\mid 1\le i<j\le n\}\cup\{2e_i\mid 1\le i\le n\}$. 
$W_K=S_n$ which acts on $\mathfrak t_\mathbb R\cong\mathbb R^n$ by permuting the coordinates. 

Let $x\in \mathfrak t_\mathbb R$ be non-zero. 
We assume, as we may, that $x_1\le \cdots\le x_{n}.$

Suppose that $R^+=R^+(\mathfrak q)=0.$  If $x_n>0, $ then $\psi_n=2e_n\in \mathcal R^+.$
So $x_n\le 0$.  If $x_n<0,$ then $x_j<0~\forall j, 1\le j\le n.$  It follows that 
   $x_i+x_j<0$ for any $i,j, 1\le i< j\le n.$  Therefore $R^-=n(n+1)/2.$

Now assume that $x_n= 0.$ Since $x\ne 0$, $x_1<0$.  Let $r<n$ be the largest positive integer such that $x_r<0.$ 
Then $-(e_i+e_j)\in \mathcal R^-$ when $1\le i\le r, 1\le j\le n$ and $i<j$. Moreover 
$-(e_i+e_j)\notin \mathcal R^-$ if $r<i< j\le n$. So we obtain that $R^-=r(2n-r+1)/2$.

{\em Thus $R^+=0$ implies that $R^-=r(2n-r+1)/2$, where $r$ is the largest integer in the range $1\leq r\leq n-1$ such that $x_{r}<0$.}

Next, suppose that $R^+=R^+(\mathfrak q_x)=1.$  Then $x_n>0$, as otherwise $x_j\le 0 ~\forall j$ 
which forces $R^+=0.$  This implies that $2e_n\in \mathcal R^+.$  Hence $2e_j\notin \mathcal R^+$
and so $x_j\le 0~\forall j<n.$  Now $x_{n-1}\le -x_n<0$ as otherwise $e_{n-1}+e_n\in 
\mathcal R^+.$   It follows that, for any $j<n$,  $x_j<0$ and so $-2e_j\in \mathcal R^-~\forall j<n.$ 
Also $-(e_i+e_j)\in \mathcal R^-$ for $1\le i<j<n$.  Finally, let $1\le s<n$ be the largest integer 
such that $x_s<-x_n$ if exists. Then $-(e_j+e_n)\in \mathcal R^-$ if and only if $1\le j\le s.$
Therefore $R^-=s+n(n-1)/2$. If $x_j\ge -x_n$ for all $1\leq j\leq n-1$, then $R^-=n(n-1)/2$. 
\subsection{ Type DIII} Let $G=SO^*(2n), K=U(n).$  We assume that $n\ge 4.$\\
$\Sigma=\{\psi_i=e_i-e_{i-1}\mid 1\le i<n\}\cup\{\psi_n=e_{n-1}+e_n\}$, where $\psi_n$ is the non-compact simple root.\\
$\Delta^+_K=\{e_i-e_j\mid 1\le i<j\le n\}$,\\
$\Delta^+_{\textrm{n}}=\{e_i+e_j\mid 1\le i<j\le n\}$,\\
$W_K=S_n$ which acts on $\mathfrak t_\mathbb R\cong \mathbb R^n$ by permuting the coordinates.

Let $x\in \mathfrak t_\mathbb R$, $x\ne 0$.  We assume that $x_1\le \cdots\le x_n.$

$(a)$ Suppose that $R^+=R^+(\mathfrak q_x)=0.$ 

{\em Case 1:} \underline{Suppose that $x_n< 0$.}  Then $R^-=n(n-1)/2$.  

{\em Case 2:}  \underline{Suppose that $x_n= 0$.}  
Let $1\leq s\le n$ be the smallest integer such that $x_{s}= 0$. Note that $s\geq 2$ as $x\neq 0$.  
Thus $x_j=0$ for $s\leq j\le n$. Therefore 
$\mathcal R^-=\{-(e_i+e_j)\mid 1\le i\le s-1<j\le n\}\cup\{-(e_i+e_j)\mid 1\le i<j\le s-1\}$ and we have $R^-=(s-1)(n-s+1)+(s-1)(s-2)/2=(s-1)(2n-s)/2$. When $s=n,$ we obtain $R^-=n(n-1)/2$.

{\em Case 3:} \underline{Suppose that $x_n>0.$}  Then $x_{n-1}\le -x_n<0.$  If $x_{n-1}<-x_n$, then $R^-=|\Delta^-_{\textrm{n}}|$. Suppose that $x_{n-1}=-x_n.$
Let $1\leq t\leq n-1$ be the smallest integer 
such that $x_{t}=-x_{n}$.   
Then $\mathcal R^-=\{-(e_i+e_n)\mid 1\leq i\le t-1\}\cup\{-(e_i+e_j)\mid 1\le i<j\le n-1\}$ forcing $R^-=(t-1)+(n-1)(n-2)/2.$
Thus, when $R^+(\mathfrak q_x)=0, x\ne 0$,
$R^-\in \{(s-1)(2n-s)/2\mid 1\le s\le n\}\cup\{ (t-1)+(n-1)(n-2)/2\mid 1\le t\le n-1\}$. All these values of $R^-$ are realized for suitable choices of non-zero $x\in\mathfrak{t}_{\mathbb{R}}$.

$(b)$ Next, suppose that $R^+(\mathfrak q_x)=1.$  Then $\mathcal R^+=\{e_{n-1}+e_n\}$,
$x_n>0,$ and $x_j+x_n\le 0$ for all $j,  1\leq j\leq n-2$. 

Note that, if for all $j$, $1\leq j\leq n-2$, we have $x_{j}+x_{n}<0$, then as $x_{n-1}\leq x_{n}$, $\mathcal{R}^{-}=\{-(e_{i}+e_{j})\mid 1\leq i<j\leq n-2\}\cup \{-(e_{i}+e_{n})\mid 1\leq i\leq n-2\}\cup \{-(e_{i}+e_{n-1})\mid 1\leq i\leq n-2\}$. Hence in this case $R^{-}(\mathfrak{q}_{x})=(n-2)(n+1)/2$. 
Otherwise let $s$ be the smallest integer such that $1\leq s\leq n-2$ and $x_{s}=-x_{n}$.

Now if $x_{n-1}<x_n,$
 then $\mathcal R^-=\{-(e_i+e_j)\mid 1\le i<j\le n-1\}\cup \{-(e_i+e_n)\mid 1\le i\le s-1\}$ and so $R^-= (s-1)+(n-1)(n-2)/2.$ Finally if $x_n=x_{n-1},$  then $\mathcal R^-=\{ -(e_i+e_j)\mid 1\le i<j\le n-2\}
 \cup \{ -(e_i+e_k)\mid 1\le i\le s-1, k=n-1,n\}$ and we have $R^-=2(s-1)+(n-2)(n-3)/2$.
 
 In summary, if $R^+(\mathfrak q_x)=1$, then
   $R^-\in \{(s-1)+(n-1)(n-2)/2\mid 1\leq s\leq n-2\}\cup \{2(s-1)+(n-2)(n-3)/2\mid 1\leq s\leq n-2\}\cup\{(n-2)(n+1)/2\}$ and all these values of $R^{-}$ are realized for suitable choices of non-zero $x\in\mathfrak{t}_{\mathbb{R}}$.
\subsection{ Type EIII}  Here $G\subset E_{6,\mathbb{C}}$, where $E_{6,\mathbb{C}}$ is the simply connected complex Lie group with Lie algebra $\mathfrak{g}=\mathfrak{e}_{6}$, the exceptional complex simple Lie algebra (having rank 6) and  
$Lie(G)=\mathfrak g_0=\mathfrak e_{(6,-14)}$. Moreover the maximal compact $K\cong Spin(10)\cdot U(1)$ of $G$ has Lie algebra $\mathfrak k_0\cong 
\mathfrak{so}(10)\oplus \mathbb R.$\\   $\mathfrak t_\mathbb R^* $ is the subspace 
of $\mathbb R^8$ orthogonal to the vectors $e_6-e_7, e_7+e_8.$ Thus 
$x=(x_1,\cdots, x_8)\in \mathfrak t_\mathbb R$ if and only $x_6=x_7=-x_8$.\\
$\Sigma =\{\psi_i\mid 1\le i\le 6\}$, with the non-compact simple root being $\psi_1$, where 
\[\psi_1=(1/2)(e_1-e_2-e_3-e_4-e_5-e_6-e_7+e_8),\]
\[\psi_2=e_2+e_1, 
\psi_3=e_2-e_1, \psi_4=e_3-e_2, \psi_5=e_4-e_3, \psi_6=e_5-e_4.\]  
$\Delta_K^+=\{e_j\pm e_i\mid 1\le i<j\le 5\}.$\\
$\Delta^+_{\textrm{n}}=\{(1/2)(-e_6-e_7+e_8+
\sum_{1\le j\le 5}(-1)^{a_j}e_j)\mid \sum_{1\leq j\leq 5}a_j\equiv 0\!\pmod 2\}
$.

Let $A\subset \{0,1\}^5$ be the set consisting of elements $a=(a_1,\ldots, a_5)
$ with $\sum_{1\le j\le 5}a_j$ even. The set $\Delta^+_{\textrm{n}}$ is in bijective correspondence with $A$ where $a\in A$ corresponds to the root $\gamma_a$ defined as 
\[\gamma_a:=(1/2)(-e_6-e_7+e_8+\sum_{1\le j\le 5} (-1)^{a_j}e_j).\]
For $x\in \mathfrak t_\mathbb R$, define 
\[ a(x):=\sum_{1\le j\le 5}(-1)^{a_j}x_j.\]
Thus $\gamma_a(x)=(-3x_6+a(x))/2$. 

We also view $A$ as a {\it subgroup} of $\mathbb Z_2^5$ in the obvious manner. 
Then $W_K=A\rtimes S_5$ where $S_5$ acts by permuting the coordinates $x_1, x_2, x_3,x_4, x_5$ of $x\in 
\mathfrak t_\mathbb R$ and $a\in A$ replaces the coordinate $x_i$ of $x$ by $(-1)^{a_i} x_i$ for $ 1\le i\le 5.$ 

 Fix $x\in \mathfrak t_\mathbb R$, $x\ne 0$.  Replacing by an element in its $W_K$-orbit
we may (and do) assume that $x_1\ge x_2\ge x_3\ge x_4\ge |x_5|\geq 0.$ 
If $a\in A,$ then 
\begin{equation}\label{e6-negative-root-criterion}
\gamma_a\in \mathcal R^- (\mathfrak q_x)\iff \gamma_a(x)<0\iff   
3x_6>a(x).
\end{equation}
\begin{lemma}\label{e6-hodge0q}
Let $x\in \mathfrak t_\mathbb R, x\ne 0.$  Suppose that $R^+(\mathfrak q_x)=0$.
Then  
\begin{equation}R^-(\mathfrak q_x)\in \{8,11,12,13,14,15,16\}.
\end{equation} Moreover each of these values are 
attained for suitable choices of non-zero $x$ in $\mathfrak{t}_{\mathbb{R}}$.
\end{lemma}
\begin{proof} Let $x\in \mathfrak t_\mathbb  R$ be non-zero.  Assume that $R^+(\mathfrak q_x)=0.$  Thus $\gamma_a(x)\le 0$ or equivalently $a(x)\leq 3x_6$ for all $a\in A$.
 We may (and do) assume that $x_1\ge x_2\ge x_3\ge x_4\ge |x_5|.$  We shall denote by $\sigma(x) =\max_{a\in A} a(x)=\sum_{1\le j\le 5} x_j.$
Note that $\gamma_a=\gamma_b\in \Delta^+_{\textrm{n}}$ if and only if $a=b.$ On the other hand If $\sigma(x)<3x_6$, then $R^-(\mathfrak{q}_{x})=2^4=16$. Thus from now on assume that $\sigma(x)=3x_6$. 
 Define $s\ge 0$ to be the number of positive numbers among $x_1,x_2,x_3, x_4.$ Note that as $x\neq 0$, $s\geq 1$. 
 
 {\it Case 1:} \underline{Assume that $x_5\geq 0$}. First assume that $x_5=0$.  
 If $c\in A$, then $c(x)=\sigma(x)$ if and only if $c_j=0$ whenever $1\le j\le 4$ and $x_j>0$.  
 The number of distinct $c\in A$ such that $c(x)=\sigma(x)$ equals $2^{4-s}$. However if $x_5>0$, then $s=4$ and the number of distinct $c\in A$ such that $c(x)=\sigma(x)$ is $2^{4-4}=1$ and the only such $c$ is zero. Thus in this case $R^-(\mathfrak{q}_{x})\in\{16-2^{4-s}\mid 1\leq s\leq 4\}=\{8,12,14,15\}.$ 

{\it Case 2:} \underline {Assume that $x_5<0$}.  Then $s=4$.   
 Let $t=|\{a\in A\mid a(x)=\sigma(x)\}|$. If $x_1\geq x_2\geq x_3\geq x_4>-x_5$, then $t=1$, if $x_1\geq x_2\geq x_3>x_4=-x_5$, then $t=2$, if $x_1\geq x_2>x_3=x_4=-x_5$, then $t=3$, if $x_1>x_2=x_3=x_4=-x_5$, then $t=4$ and finally if $x_1=x_2=x_3=x_4=-x_5$, then $t=5$. Thus $R^-(\mathfrak{q}_{x})\in\{11,12,13,14,15\}$.

 It is clear from the above proof that each value of $R^-$ as asserted in the statement of this lemma are in fact realized for suitable choices of non-zero $x$ in $\mathfrak{t}_{\mathbb{R}}$. 
\end{proof} 
 Next we consider $\theta$-stable parabolic subalgebras $\mathfrak q_x$ with $R^+(\mathfrak q_x)=1$ and determine the possible values of $R^-(\mathfrak q_x).$
\begin{lemma} \label{e6-hodge1q}   Let $x\in\mathfrak{t}_{\mathbb{R}}$, $x\neq 0$. Suppose that $R^+(\mathfrak q_x)=1.$
   Then:
   \[R^-(\mathfrak q_x)\in\{5,9,11, 12,13,14,15\}.\] Moreover each of these values are attained for suitable choices of non-zero $x$ in $\mathfrak{t}_{\mathbb{R}}$.
\end{lemma}
\begin{proof}  We keep the notations of Lemma $(\ref{e6-hodge0q})$.
Let $x\in \mathfrak t_\mathbb R$, non-zero be such that $R^+(\mathfrak q_x)=1$. 
As before, we assume that $x_1\ge x_2\ge x_3\ge x_4\ge |x_5|.$ 
It is readily seen that $\gamma_0=
(1/2)(e_1+e_2+e_3+e_4+e_5-e_6-e_7+e_8)\in \mathcal R^+$. 
Note that $\gamma_0(x)=(-3x_6+\sigma(x))/2$.  

    Since $R^+=1$, if $a\in A$ is non-zero, then $\gamma_a(x)\le 0$.    Thus 
    $a(x)\le 3x_6<\sigma(x)~\forall a\ne 0$.  It follows that $x_j\ne -x_5$
for $j<5$, and, in particular, $x_4>0$ if $x_5= 0$. We have $\mathcal R^-=\{-\gamma_a\mid a(x)<3x_6, a\in A\}$. 

    {\it Case 1:} \underline{Suppose $x_5\ge 0.$}   If $a(x)<3x_6$ for all non-zero $a\in A$, then $R^-=15.$
    
    Suppose that $a(x)=3x_6$ for some non-zero $a\in A$. Note that if $x_1\ge x_2\ge x_3>x_4\ge x_5$, then $t=1$, if $x_1\ge x_2>x_3=x_4>x_5$, then $t=2$, if $x_1\geq x_2>x_3=x_4=x_5>0$, then $t=3$, if $x_1>x_2=x_3=x_4>x_5$, then $t=3$, if $x_1>x_2=x_3=x_4=x_5>0$, then $t=6$, if $x_1=x_2=x_3=x_4>x_5$, then $t=4$ and finally if $x_1=x_2=x_3=x_4=x_5>0$, then $t=10$. Therefore $R^-\in\{5,9,11,12,13,14,15\}$ and clearly all these values are realized for suitable choices of non-zero $x$ in $\mathfrak{t}_{\mathbb{R}}$.

{\it Case 2:} \underline{Suppose that $x_5<0.$}  If $a=(0,0,0,1,1)$, then $\gamma_a(x)\ge \gamma_b(x)$ 
     for all $b\in A, b\ne 0.$  If $\gamma_a(x)<0,$ then $R^-=15$.  
     So assume that 
     $\gamma_a(x)=0$.  Then the number of $b\in A$ for which $\gamma_b(x)=\gamma_a(x)=0$ 
     equals $t:=|\{j\le 4\mid x_j=x_4\}|.$  For the remaining non-zero $c\in A$ we have 
     $\gamma_c\in \mathcal R^-.$ Thus 
     $R^-=15-t\in \{11, 12, 13,14\}$. 
     Again, each of the values $11, 12, 13, 14,$ is readily seen to be realized for suitable non-zero $x\in\mathfrak{t}_{\mathbb{R}}$, completing the proof.
     \end{proof}
     \subsection{Type EVII} Let $G\subset E_{7,\mathbb C}$ be the real form of the simply connected complex Lie group corresponding to the real form $\mathfrak g_0:=\mathfrak e_{(7,-25)}$ of the exceptional complex Lie algebra $ \mathfrak e_7.$ 

$K$ is locally isomorphic to $U(1)\times E_6$;  $\mathfrak k_0\cong \mathbb R\times \mathfrak e_{(6,-78)}$.   \\
$\mathfrak t_\mathbb R^*\subset \mathbb R^8$ is the space orthogonal to $e_7+e_8.$ Thus $\mathfrak t_\mathbb R\subset \mathbb R^8$ consists of vectors $x=(x_1,\ldots, x_7, x_8)$ where $x_7=-x_8$.

$\Sigma=\{\psi_i\mid 1\le i\le 7\}$ where 
$\psi_1= (1/2)(e_1+e_8)-(1/2)(e_2+\cdots +e_7)$, 
 $\psi_2=e_1+e_2$, $\psi_3=e_2-e_1$, $\psi_4=e_3-e_2$, $\psi_5=e_4-e_5$, $\psi_6=e_5-e_4$ and $\psi_7=e_6-e_5$. 
The non-compact simple root is $\psi_7=e_6-e_5.$

$\Delta_K^+=\{e_j\pm e_i\mid 1\le i<j\le 5\}\cup \{(1/2)(-e_6-e_7+e_8+\sum_{1\le j\le 5} (-1)^{a_j}e_j)\mid a\in A\},$  where $A\subset \{0,1\}^5$ consists of $a=(a_1,\cdots, a_5)$
with $\sum a_j$ being even.

$W_K\cong W(E_6)$ is generated by the simple reflections $s_i:=s_{\psi_i}, 1\le i\le 6.$ 
The subgroup $W_1$ generated by $s_j$, for $2\le j\le 6$, is isomorphic to the Weyl group of $Spin(10)\cong S_5\rtimes A$ and acts 
on $\mathfrak t_\mathbb R$ by permuting the coordinates $x_1,\ldots, x_5$ of $x$ and changing the signs of even number of these coordinates.   We shall only use the action of the subgroup $W_1$
on $\mathfrak t_\mathbb R.$

$\Delta^+_{\textrm{n}}=\{(1/2)(e_6-e_7+e_8-\sum_{1\le j\le 5}(-1)^{c_j}e_j)\mid c\in A \}\cup\{e_6\pm e_j\mid 1\le j\le 5\}\cup \{e_8-e_7\}.$ We set $\gamma_j^\pm:=e_6\pm e_j$ and 
 $\beta_c:= (1/2)(e_6-e_7+e_8-\sum_{1\le j\le 5}(-1)^{c_j}e_j)$, $c\in A$.
We shall refer to $\{\pm \gamma_j^\pm\}$ as roots of {\em type} $\mathcal D$ and to $\{\pm\beta_c\}_{c\in A}$
as roots of {\em type} $\mathcal E.$  Denote by $\mathcal D^-=\{-\gamma^\pm_j\}$ (resp. $\mathcal E^-=\{-\beta_c\}$) the set 
of negative (non-compact) roots of type $\mathcal D$ (resp. $\mathcal E$).

We set $a(x):=\sum_{1\le j\le 5} (-1)^{a_j}x_j$ for $x\in \mathfrak t_{\mathbb R}$ and  
$a\in A$. Define $\sigma(x):=\max_{a\in A}a(x)$ and $\tau(x):=-\min_{a\in A} a(x).$ Note that $\sigma(x)=\tau(x)$ if $x_j=0$ for some $j\le 5$.
 \begin{lemma} \label{e7-hodge0}
Let $x\in \mathfrak t_\mathbb R$ and $x\ne 0.$
   Suppose that $R^+(\mathfrak{q}_x)=0$. Then:
   \[R^-(\mathfrak q_x)\in \{17, 21, 22,23,24,25,26, 27\}.\]
   Moreover each of these values are attained for suitable choices of non-zero $x\in\mathfrak{t}_{\mathbb{R}}$.
\end{lemma}
\begin{proof}
    The vanishing of $R^+(\mathfrak q)=0$ implies that for any positive non-compact 
    root $\gamma$, we have $\gamma(x)\le 0$.  

    Using the action of $W_1\cong S_5\rtimes A\subset W_K$, without 
    loss of generality we assume that $x_1\ge x_2\ge x_3 \ge x_4\ge |x_5|$.  So $\sigma(x)
    =\sum_{1\le j\le 5}x_j$ and $\tau(x)=x_1+x_2+x_3+x_4-x_5$.  
    Note that $\tau(x)=\sigma(x)$ if and only if $x_5=0$.
    
    Taking $\gamma=e_8-e_7\in \Delta^+_{\textrm{n}}$, we have $\gamma(x)=x_8-x_7\le 0.$
    Since$-x_8=x_7$, $e_7-e_8\in \mathcal R^-$ if and only if $x_7>0.$
    
    Since $\gamma_j^\pm (x)=x_6\pm x_j\le 0~\forall j\le 5,$ we have $-x_6\ge x_1.$
    If $-x_6>x_1$, then $\mathcal D^-\subset \mathcal R^-.$
    Since $\beta_a(x)=(x_6-2x_7-a(x))/2\le 0~\forall a\in A,$ 
    we have 
   \begin{equation}\label{e7-main-inequality}
     -x_6+2x_7\ge \tau(x)=x_1+x_2+x_3+x_4-x_5.
    \end{equation}
If the inequality is strict, then $\mathcal E^-\subset \mathcal R^-.$    

    {\em Case 1:} \underline{Suppose that $x_6=0.$}  Now, since $-x_6\ge |x_j|,$ we have $x_j=0~\forall~1\leq j\leq 6.$ So none of the roots $-e_6\pm e_j$ with $1\le j\le 5$ is in $\mathcal R^-.$
    Since $x\ne 0$ and $-x_8=x_7\ge 0$, we must have $x_7> 0$.  It follows that 
    $e_7-e_8\in \mathcal R^-$.  Since $\tau(x)=0$, strict inequality holds in Equation $(\ref{e7-main-inequality})$ and so $\mathcal E^-\subset \mathcal R^-.$  Therefore $R^-(\mathfrak q_x)=17$ in this case.

    {\em Case 2:} \underline{Suppose that $x_6<0$ and $x_7=0.$}  Then $x_8=0$ and so $e_8-e_7\notin \mathcal R^-.$   Suppose that $-x_6>\tau(x).$  Then $\mathcal E^-\subset \mathcal R^-$.
    Also $-x_6>\tau(x)\ge x_1$ and so $\mathcal D^-\subset \mathcal R^-.$ Thus $R^-=26.$

    For the rest of the case assume that $-x_6=\tau(x).$ Then 
    $-x_6\ge x_1+x_2$ (note that equality holds {\em only} if $x_5=0$).  
    If $x_2>0$ or if $-x_6>x_1$, 
    then  $\mathcal D^-\subset \mathcal R^-$. 
    Suppose that $x_2=0$.  Then $R^-=17 $ or $26$ according as $-x_6=x_1$ or $-x_6>x_1$.
 
    Suppose that $x_2>0$.  Let $t=|\{j\le 5\mid x_j=|x_5|\}|$.
    If $x_5=0$,   
    then $\beta_c(x)=0$ for $2^{t-1}$ members $c\in A$ and for the 
    rest of the members $c\in A$, we have $-\beta_c\in \mathcal R^-$. 
    If $x_5>0$, then $\tau(x)=-c(x)$ for precisely $t$ distinct members $c$ of $A$ and for the rest of the members $c\in A$, $-\beta_{c}\in\mathcal{R}^{-}$.
    So $|\mathcal E^-\cap \mathcal R^-|=16-t.$
        If $x_5<0$, then $-\tau(x)=c(x)<a(x)~\forall~a\in A,~a\neq c$ with $c=(1,1,1,1,0)$,
    and so $|\mathcal E^-\cap \mathcal R^-|=15.$
 
   Thus summarizing, we get that under the blanket assumption $-x_6=\tau(x)$, $R^-\in\{22,24,25\}$ if $x_5=0$, $R^-\in\{21,22,23,24,25\}$ if $x_5>0$ and finally $R^-=25$ if $x_5<0$.

{\em Case 3:} \underline{Suppose that $x_6<0$ and $x_7>0$.} Then $-e_8+e_7\in \mathcal R^-$. 
    Let $s:=|\{j\le 5\mid -x_6=|x_j|\}|$ so that $|\mathcal D^-\cap \mathcal R^-|=10-s.$ 
    Suppose that $-x_6+2x_7>\tau(x),$ then $\mathcal E^-\subset \mathcal R^-.$ Thus $R^-\in\{22,23,24,25,26,27\}$. 
    
    For the rest of the case assume that $-x_6+2x_7=\tau(x).$
    Let $t:=|\{j\le 5\mid x_j=0\}|$. Then $|\mathcal E^-\cap \mathcal R^-|=16-2^{t-1}$ if $t>0$.
    Since $\tau(x)>-x_6>0,$ we have $t\le 4$.  We claim that $t\le 3$. For, if $t=4,$ then $-x_6+2x_7=x_1$ which implies that $x_6+x_1=2x_7>0$ and so $e_6+e_1\in \mathcal R^+.$ This contradicts 
     our hypothesis that $R^+=0$. 

We need to consider three subcases: 

    {\em Subcase $(i)$$:$} \underline{Suppose that $x_5< 0$}.  Then $t=0$ 
    and so  $|\mathcal E^-\cap \mathcal R^-|= 15.$ Since $|\mathcal D^-\cap\mathcal R^-|=10-s$ and 
    $e_7-e_8\in \mathcal R^-$ we conclude that $R^-\in\{21,22,23,24,25,26\}.$

    {\em Subcase $(ii)$$:$}
    \underline{Suppose that $x_5=0$}.  Then $1\le t\leq 3$ and $|\mathcal E^-\cap \mathcal R^-|=16-2^{t-1}.$  
    In this case $R^-=27-2^{t-1}-s$. Note that $s+t\le 5.$ So $R\in \{21,22, 23,24,25,26\}.$

    {\em Subcase $(iii)$$:$}  \underline{Suppose that $x_5>0$}. We have 
    $\tau(x)=x_1+x_2+x_3+x_4-x_5.$   As $-x_6+2x_7=\tau(x)$,  
    $|\mathcal E^-\cap \mathcal R^-|=16-r$ where $ r=|\{j\le 5\mid x_j=x_5\}|$.  
    As before,  set $s:=|\{j\le 5\mid -x_6=x_j\}|$ so that $|\mathcal D^-\cap \mathcal R^-|=10-s.$  
    Note that if $1\le r<5$, then $-x_6>x_5$ 
     and hence $0\le s\le 5-r.$    
    If $r=5$, then $s=5$ or $0$ according as $-x_6=x_1=x_5$ or $-x_6>x_1=x_5$. Therefore
    $1\le r+s\le 5$ or $r+s=10$. 
    It follows that $R^-=27-(r+s)$ and the possible values 
    of $R^-$ are $17, 22, 23,24, 25,26.$ Also it is clear that $R^-$ attains each of the values $17,21,22,23,24,25,26,27$ for suitable choices of non-zero $x\in\mathfrak{t}_{\mathbb{R}}$. 
This completes the proof.
\end{proof}
\begin{lemma}\label{e7-hodge1}
Suppose that
$R^+(\mathfrak q_x)=1$ where $x\in \mathfrak t_\mathbb R$ and $x\ne 0.$ Then:  
\[R^-(\mathfrak q_x)\in\{10,18,21,22,23,24,25,26\}.\] 
Moreover each of these values are attained for suitable choices of non-zero $x\in\mathfrak{t}_{\mathbb{R}}$.
\end{lemma}
\begin{proof}  
 We may (and do) assume that $x_1\ge x_2\ge x_3\ge x_4\ge |x_5|$.
    
    {\em Case 1:} \underline{Suppose that $e_8-e_7\in \mathcal R^+.$}  Then $x_8=-x_7>0$  
    and $e_7-e_8\notin \mathcal R^-.$
    Since $\beta_c\notin \mathcal R^+ ~\forall c\in A,$
     Equation $(\ref{e7-main-inequality})$ holds.
         Since $x_7<0$ we have $-x_6>\tau(x)\ge x_1.$ Therefore $\mathcal D^-\subset \mathcal R^-.$
    If $-x_6+2x_7>\tau(x)$, then $\mathcal E^-\subset \mathcal R^-$ and so
$R^-=26.$  

Suppose that $-x_6+2x_7=\tau(x).$   If $\tau(x)=0$, then $x_1=0$ which implies that $x_j=0$ for all $j\le 5.$
So $|\mathcal E^-\cap\mathcal R^-|=0$ and $\mathcal D^-\subset \mathcal R^-.$ It follows that  $R^-=10$.

Suppose that $-x_6+2x_7=\tau(x)>0.$ As noted already, $\mathcal D^-\subset \mathcal R^-$,  the value of $|\mathcal E^-\cap \mathcal R^-|$ depends on the sign of $x_5$. Thus there are three cases to consider:

{\em Subcase $(i)$$:$} \underline{Suppose that $x_5<0$}.  Then there is a unique $c\in A$ such that $\tau(x)=-c(x)$. So 
$-\beta_{c}(x)=0$ and if $a\in A$ with $a\neq c$, $-\beta_a\in \mathcal R^-.$  So $ R^-=25.$

{\em Subcase $(ii)$$:$} \underline{Suppose that $x_5=0$}.   Let $t=|\{ j\le 5\mid x_j=0\}|$. 
Then $|\mathcal E^-\cap \mathcal R^-|=16-2^{t-1}$.   
Since $\tau(x)>0$ we have $t\le 4$. Thus $R^-\in \{18, 22, 24,25\}.$

{\it Subcase $(iii)$$:$}  \underline{Suppose that $x_5>0$}.  Let $r=|\{1\leq j\le 5\mid x_j=x_5\}|$.  There are 
$r$ distinct elements $c\in A$ such that $ \tau(x)=-c(x)$ and we have $|\mathcal E^-\cap\mathcal R^-|=16-r.$ 
Therefore $R^-=26-r$, $1\leq r\leq 5$. Thus $R^{-}\in\{21,22,23,24,25\}$.

{\em Case 2:} \underline{Suppose that $e_6+e_i\in  \mathcal R^+$ for some $i$.} 
Since $x_6+x_1\ge x_6+x_i>0$ and since $R^+=1$, we must have $i=1$.  Moreover $x_6\pm x_j\le 0$ 
if $1<j\le 5.$  So
$x_1>-x_6\ge x_2\ge 0$. It follows that $-(e_6-e_1)\in\mathcal R^-.$  
Let $s=|\{2\le j\le 5\mid -x_6=|x_j|\}|$.
Then $|\mathcal D^-\cap \mathcal R^-|=9-s$ if $0\leq s\leq 3$. if $s=4$, then $|\mathcal D^-\cap \mathcal R^-|=5$ or $1$ according as $x_5\ne 0$ or $x_5=0$ respectively.  
Also $x_7=-x_8\ge 0$ as $e_8-e_7\notin \mathcal R^+$.  

Since $\mathcal E^-\cap \mathcal R^+=\emptyset$
we have $-x_6+2x_7\ge \tau(x).$ If the inequality is strict, we have $\mathcal E^-\subset 
\mathcal R^-.$ 
Also note that if $-x_6+2x_7>\tau(x)$, then $x_7>0$ and so $e_7-e_8\in \mathcal R^-$.   Thus $R^-\in\{18,22,23,24,25,26\}$ when $ -x_6+2x_7>\tau(x)$.

Now assume that $-x_6+2x_7=\tau(x)$. Thus 
$2x_7=(x_6+x_1)+x_2+x_3+x_4-x_5\ge x_6+x_1>0.$
In particular $x_7>0$ and so $e_7-e_8\in 
\mathcal R^-$. Again the value of $|\mathcal E^-\cap \mathcal R^-|$ depends on the sign of $x_5$ and hence there are three cases to consider:

{\em Subcase $(i)$}: \underline{Suppose $x_5<0$}.   Then $|\mathcal E^-\cap \mathcal R^-|=15$ and hence $R^-\in\{21,22,23,24,25\}$. 

{\em Subcase $(ii)$}:  \underline{Suppose $x_5=0$}.  Let $t:=|\{2\le j\le 5\mid x_j=0\}|$. Then $1\le t\le 4$ and there are $2^{t-1}$ distinct elements $c\in A$ such that $-c(x)=\tau(x)$. 
For each such $c$, we have the equality $-x_6+2x_7=-c(x)$ and so $\beta_c(x)=0.$   
For the remaining elements $a\in A,$ we have 
$-\beta_a(x)>0$ and so $|\mathcal E^-\cap \mathcal R^-|=16-2^{t-1}.$  
Hence $R^-\in\{26-2^{t-1}-s\mid 1\leq t\leq 4,~0\leq s\leq 4-t\}$ if $x_6\neq 0$. If $x_6=0,$ then 
$x_2=0$, $t=s=4$, and $|\mathcal D^-\cap\mathcal R^-|=1$, which implies that $R^-=10$. Therefore $R^-\in \{10,18,21,22,23,24,25\}.$

{\em Subcase $(iii)$}: \underline{Suppose $x_5>0$}.  Let $r=|\{1\leq j\le 5\mid x_j=x_5\}|.$  Then $1\le r\le 4$ and we have $\tau(x)=-c(x)$
for $r$ distinct elements $c\in A$. For the remaining $a\in A, -\beta_a\in \mathcal R^-$. 
Thus $|\mathcal E^-\cap \mathcal R^-|=16-r.$
If $r=4,$ then $s=0$ or $4$ according as $-x_6>x_2$ or $-x_6=x_2$ respectively.  
Consequently, 
$|\mathcal D^-\cap \mathcal{R}^-|$ equals $9$ or $5$ according as $(r,s)=(4,0)$ or $(4,4)$ respectively. If $1\leq r\leq 3,$ then $R^-\in\{26-(r+s)\mid 0\leq s\leq 4-r, 1\leq r\leq 3\}$. Therefore we have  
$R^-\in \{18,22,23,24,25\}$.  

{\em Case 3:} \underline{Suppose that $\gamma_j^-\in \mathcal R^+.$}  
Since $R^+=1$ and $0<x_6-x_j\le x_6-|x_5|$, we must have $j=5$ and $x_5<x_6\le -x_5$. Now $x_6+x_1\ge x_6-x_5>0$ implies that $e_6+e_j\in \mathcal R^+$ for $1\le j\le 4$ a contradiction.
So, this case does not occur at all.

{\em Case 4:}  \underline{Suppose that $\beta_b\in \mathcal R^+$ for some $b\in A.$}
Then $-\tau(x)=b(x)<x_6-2x_7\leq a(x)$ for all $a\in A$ with $a\neq b$ where $b=(1,1,1,1,0)$. This implies that $x_4>x_5$. 
Setting $k:=|\{1\leq j\leq 4 \mid x_j=x_4\}|$,  we have $1\leq k\leq 4$.
Fix $c:=(1,1,1,0,1)\in A$ so that $c(x)=min_{a\in A,a\neq b}a(x)$.

If $x_6-2x_7<c(x)$, then $|\mathcal{E}^{-}\cap \mathcal{R}^{-}|=15$, while assuming that $x_6-2x_7=c(x)$, if either $x_5\geq 0$ or $x_4>-x_5>0$, then $|\mathcal{E}^{-}\cap \mathcal{R}^{-}|=15-k$, whereas if $x_4=-x_5>0,$ it is straightforward to see that $|\mathcal{E}^{-}\cap \mathcal{R}^{-}|=14,12,9,5$ according as $k=1,2,3,4$ respectively.

Since $e_6\pm e_j\notin \mathcal R^+$ we have $-x_6\ge |x_j|$ for $1\le j\le 5$.
Set $s:=|\{1\leq j\leq 5\mid -x_6=|x_j|\}|$. 
If $0\leq s\leq 4$, then $|\mathcal{D}^{-}\cap\mathcal{R}^{-}|=10-s$.  If $s=5$, then $|\mathcal D^-\cap \mathcal R^-|=5$ or $0$ according as $x_5\ne 0$ or $x_5=0$. 

Since $e_8-e_7\notin \mathcal R^-$ and since $x_7=-x_8$, we have 
$x_7\ge 0$.  So $e_7-e_8\in \mathcal R^-$ if and only  
We observe that $x_7>0$ if $x_i=0$ for $1\le i\le 6$. The values of $k$, $s$ are not independent of each other. Also 
since $\beta_c(x)\ge 0$, depending on the specific values 
of $k,s,$ we are led to the conclusion that $x_7$ is positive or that $x_7$ vanishes.  
The following subcases cover all the possibilities.

{\it Subcase $(i)$$:$} \underline {Assume that $\beta_c(x)=0$ and that $x_4=-x_5>0$}. 
Let $k=1$. Then $0\leq s\leq 3$. If $s=0$, then $R^{-}=24$ or $25$ according as $x_7=0$ or $x_7>0$ respectively.
If $1\leq s\leq 3,$ then $x_7>0$ and $R^-=24,23,22$ according as $s=1,2,3$ respectively.

When $k=2$, we have $0\leq s\leq 2$. If $s=0$, then $R^{-}=22$ or $23$ according as $x_7=0$ or $x_7>0$ respectively.
If $1\leq s\leq 2,$ then $x_7>0$ and $R^-=22,21$ according as $s=1,2$ respectively.

When $k=3$, we have $s=1$ in view of the fact that $x_7\geq 0$.  Now we must have $x_7=0$ and consequently $R^-=18$.

When $k=4$, we have $s=5$ using $x_7\geq 0$.  In turn, this implies that $x_7=0$ 
and so 
$R^- = 10.$ 

{\it Subcase $(ii)$$:$} \underline{Assume that $\beta_c(x)=0$, and, either $x_5\geq 0$ or $x_4>-x_5>0$}. 
Then, $R^-\in \{21+j\mid 0\le j\le 5-k\}$ for $1\le k\le 3$.
If $k=4$, then either $s=0$  or $s=4$. If $(k,s)=(4,4)$, then $x_7>0$. Thus, when $k=4$, $R^-\in\{ 18,21,22\}$. 

So we see that, in this subcase, $R^-\in\{18,21,22,23,24,25\}$. 

{\it Subcase $(iii)$$:$} \underline{Assume that $x_6-2x_7<c(x)$}. Then it is already observed before that $|\mathcal{E}^-\cap \mathcal{R}^-|=15$. We simply note that the case $s=5$ and $x_5=0$ does not occur, while if $s=5$ and $x_5\neq 0$, then $x_7>0,$ forcing $R^-=21$.  Thus $R^-\in\{21,22,23,24,25,26\}$.

Finally, it is evident from the above arguments, that $R^-$ attains  each of the values $\{10,18,21,22,23,24,25,26\}$
for suitable choices of non-zero $x\in \mathfrak t_{\mathbb R}$.
This completes the proof.
\end{proof}
\subsection{Summary of results for the exceptional types:}
Below we tabulate the values of $R^-$ (each of which is realized for some non-zero  $x\in\mathfrak{t}_{\mathbb{R}}$) when $R^+=0,1$, for $G$ corresponding to {\bf EIII} and {\bf EVII}:
\[ \begin{array}{|c|c|c|}
\hline 
\mathrm{Type~of~} G& R^+=0, R^-(\mathfrak q_x)  & R^+=1, R^-(\mathfrak q_x)\\
\hline\hline
{\bf EIII}& 8,11,12,13,14,15,16& 5,9,11, 12,13,14,15\\
\mathrm{Ref:} & \mathrm{Lemma~ (\ref{e6-hodge0q})} & \mathrm{Lemma~} (\ref{e6-hodge1q})\\
\hline 
{\bf EVII} & 17, 21, 22,23,24,25,26, 27& 10,18,21,22,23,24,25,26\\
\mathrm{Ref}:& \mathrm{Lemma~} (\ref{e7-hodge0})& \mathrm{Lemma~} (\ref{e7-hodge1})\\
\hline
\end{array}
\]
\begin{center}
    {\bf Table 2:} Values of $R^-(\mathfrak q_x), x\ne 0$, when $R^+=R^+(\mathfrak q_x)=0,1$ for types {\bf EIII} and {\bf EVII}. 
\end{center}
\subsection{Summary of results for the classical types:}
Below we tabulate the values of $R^-$ (each of which is realized for some non-zero $x\in\mathfrak{t}_{\mathbb{R}}$) when $R^+=0,1$, for $G$ corresponding to each classical type:
\[
\begin{array}{|c|c|}
\hline
\mathrm{Type~of~} G &  R^-(\mathfrak q_x)~ {\textrm{when}} ~ R^+(\mathfrak q_x)=0, x\ne 0  \\
\hline\hline
{\bf AIII}~ & \\
SU(1,n), n>1 & [1,n]\cap \mathbb N.\\
\hline
{\bf AIII}~ & \\
SU(m,n),2\le m\le n & rn+(m-r)t, 0\le r\le m-1, 0\le t\le n, \\
 & (r,t)\ne (0,0). \\
\hline
   {\bf BDI}~  &  m-1, m, \cdots, 2m-2.\\
   SO_0(2,2m-2),m\ge 3&\\
 \hline
 {\bf BDI}~ & m, m+1, \cdots, 2m-1.\\
 SO_0(2,2m-1),m\ge 2&\\
 \hline
 {\bf CI}~ &  r(2n-r+1)/2, 1\le r\le n.\\
 Sp(n,\mathbb R),n\ge 2&\\
 \hline 
 {\bf DIII}~&  (s-1)(2n-s)/2,1\le s\le n;\\
 SO^*(2n), n\ge 4& (t-1)+(n-1)(n-2)/2\mid 1\le t\le n-1.\\
 \hline
\end{array}
\]
\begin{center}
    {\bf Table 3:} Values of $R^-(\mathfrak q_x),x\ne 0,$ when $R^+(\mathfrak q_x)=0$ in the classical types. 
\end{center}
\[
\begin{array}{|c|c|}
\hline
\mathrm{Type~of~} G & R^-(\mathfrak q_x)~
{\textrm{when}} ~ R^+(\mathfrak q_x)=1, x\ne 0\\
\hline\hline
 {\bf AIII}& \\
 SU(1,n) & [0,n-1]\cap \mathbb N.\\
 \hline
{\bf AIII}& \\
SU(m,n),2\le m\le n & mn+r-s,\\
& 1\le r\le m-1, m+2\le s\le m+n.\\
\hline
   {\bf BDI}  &1, [m-1,2m-3]\cap \mathbb N. \\
   SO_0(2,2m-2,m\ge 4)& \\
 \hline
 {\bf BDI} &1, [m, 2m-2]\cap \mathbb N. \\
 ~SO_0(2,2m-1),m\ge 2&\\
 \hline
 {\bf CI}~&  s+n(n-1)/2,\\
 Sp(n,\mathbb R), n\ge 3&  0\le s<n.\\
 \hline 
 {\bf DIII}~ & (s-1)+(n-1)(n-2)/2,\\ 
 SO^*(2n), n\ge 4&2(s-1)+(n-2)(n-3)/2, ~(n-2)(n+1)/2 \\
 & 1\le s\le n-2.\\
 \hline
\end{array}
\]
\begin{center}
{\bf Table 4:} Values of $R^-(\mathfrak q_{x})$, $x\neq 0$, when $R^+(\mathfrak q_{x})=1$ in the 
classical types.
\end{center}

We have the following vanishing theorem, whose proof follows from the 
Matsushima isomorphism, thanks to Table 2, Table 3 and Table 4. (See Remark $(\ref{vanishingcohomology}))$.

\begin{theorem}\label{vanishing}
Suppose that $G$ is simple so that $X=G/K$ is an irreducible Hermitian symmetric space. Let $\Gamma \subset G$ is a torsion free and uniform lattice. Then: \\
$(i)$ Let $q>0$.  $H^{0,q}(X_\Gamma)=0$ if $q$ does \underline{not} occur in row of type $G$ 
of Table 3  when $G$ is of classical type, and, in row of type $G$, column $R^+=0$ of Table 2 when $G$ is of exceptional type. \\
$(ii)$ Let $q>1.$  Then $H^{1,q}(X_\Gamma)=0$ 
if $q-1$ does \underline{not} occur  
in row of type $G$ of Table 3 and $q$ does \underline{not} 
occur in row of type $G$ of Table 4, when $G$ is of classical type;  if $q-1$ does \underline{not} occur  in row of type $G$, 
column $R^+=0$ and $q$ does \underline{not} occur in row of type $G$, column $R^+=1$
of Table 2, when G is of exceptional type. \hfill $\Box$
\end{theorem}
Note that as the vanishing of $H^{0,p}(X_{\Gamma})$ for $p=1,2$ implies that $Pic(X_{\Gamma})\cong H^{2}(X_{\Gamma},\mathbb{Z})$, one reads off from Theorem $(\ref{vanishing})$ the cases for the occurence of this phenomenon, but when $r_{\mathbb{R}}(G)\geq 3$, then as $\Gamma$ is already irreducible, $G$ being simple, Theorem $(\ref{picard-xgamma})$ already shows that $Pic(X_{\Gamma})\cong H^{2}(X_{\Gamma},\mathbb{Z})$.  
In the theorem below, we summarize certain cases when rank of $Pic(X_{\Gamma})=1$. 
The proof follows readily from Remark $(\ref{vanishingcohomology})(iii)$, Theorem $(\ref{vanishing})(i)$, Table 4 and column $R^{+}=1$ in Table 2. We have omitted the type {\bf BDI} since in this case there do exist 
$\theta$-stable parabolic subalgebras $\mathfrak q$  of $\mathfrak g_0$ 
with $R^+(\mathfrak q)=R^-(\mathfrak q)=1$.  
\begin{theorem}\label{hodge-1-1} We have $H^{1,1}(X_\Gamma)\cong \mathbb C$ and the rank of $Pic(X_\Gamma)$ equals $1$ $($resp. $Pic(X_\Gamma)\cong H^2(X_\Gamma,\mathbb Z)$$)$ in the following cases.   
{\bf AIII:} $SU(m,n)$, $(m,n)\ne (2,2), 2\leq m\leq n$ $($resp. $n\ge m\ge 3$$)$;
{\bf CI:} $Sp(n,\mathbb R)$, $n\ge 3$, $($resp. $n\ge 4$$)$;
{\bf DIII:} $SO^*(2n)$, $n\ge 5$ $($resp. $n\ge 4$$)$;
{\bf EIII} and {\bf  EVII:} $H^{1,1}(X_\Gamma)\cong \mathbb C$ and $Pic(X_\Gamma)\cong H^2(X_\Gamma,\mathbb Z),$ which is of rank $1$. \hfill $\Box$
\end{theorem}
{\bf Acknowledgements:} The research of both the authors was partially supported by the Infosys foundation.

{\bf Declaration of competing interest:} There is no competing interest involved with this work.

\end{document}